\documentclass[11pt]{amsart}
\usepackage{amsmath,amssymb,latexsym,dsfont}
\usepackage[left=2.6cm,right=2.6cm,top=2.7cm,bottom=2.7cm]{geometry}

    \usepackage{amsmath,amsfonts,amssymb,epsfig}
     \usepackage{graphicx}
     \usepackage{hyperref}
     \usepackage{cases}
     \usepackage{color}

\newtheorem{theorem}{Theorem}[section]
\newtheorem{lemma}{Lemma}[section]

\newtheorem{proposition}{Proposition}[section]
\newtheorem{corollary}{Corollary}[section]
\newtheorem{remark}{Remark}[section]

\newtheorem{definition}{Definition}[section]
\numberwithin{equation}{section}

  \newcommand{\beq}{\begin{equation}}
\newcommand{\eeq}[1]{\label{#1}\end{equation}}

      \newcommand{\R}{{\mathbb{R}}}

      \newcommand{\curl}{\operatorname{curl}}
      \newcommand{\dive}{\operatorname{div}}

      \newcommand{\loc}{\operatorname{loc}}
      
      \newcommand{\eps}{\varepsilon}
      \newcommand{\mR}{\mathbb{R}}
      \newcommand{\mS}{\mathbb{S}}

           \newcommand{\mup}{\mu^+}
           \newcommand{\mun}{\mu^-}

           \newcommand{\ep}{\eps^+}
           \newcommand{\en}{\eps^-}

      \newcommand{\bE}{\operatorname{\textbf{E}}}
      \newcommand{\bH}{\operatorname{\textbf{H}}}

     \newcommand{\supp}{\mbox{supp }}

      \makeatletter
      \def\@setcopyright{}
      \def\serieslogo@{}
      \makeatother

\newcommand{\cP}{{\mathcal P}}
\newcommand{\cF}{{\mathcal F}}
\newcommand{\cG}{{\mathcal G}}

\newcommand{\cH}{{\mathcal H}}
\newcommand{\cE}{{\mathcal E}}

\newcommand{\hH}{{\hat H}}
\newcommand{\hE}{{\hat E}}
\newcommand{\hJ}{{\hat J}}

\newcommand{\tbH}{\widetilde{\bf H}}
\newcommand{\tbE}{\widetilde{\bf E}}

\newcommand{\tH}{\widetilde H}
\newcommand{\tE}{\widetilde E}
\newcommand{\tJ}{\widetilde J}

\newcommand{\hy}{\hat y}
\newcommand{\hj}{\hat j}

\newcommand{\teps}{\widetilde \eps}
\newcommand{\tmu}{\widetilde \mu}

\newcommand{\sign}{\mbox{sign}}

\begin{document}

   \author[H.-M. Nguyen]{Hoai-Minh Nguyen}

\address[H.-M. Nguyen]{Department of Mathematics, EPFL SB CAMA, Station 8,  \newline\indent
	 CH-1015 Lausanne, Switzerland.}
\email{hoai-minh.nguyen@epfl.ch}


\title[The invisibility via anomalous localized resonance]{The invisibility  via anomalous localized resonance of   a source for electromagnetic waves}

   
\begin{abstract}
We study the invisibility via anomalous localized resonance of a general source for electromagnetic waves in the setting of doubly complementary media. As a result,  we show that cloaking is achieved if the power is blown up. We also reveal  a critical length for  the invisibility of a source that occurs when the plasmonic structure is complementary to an annulus of constant, isotropic  medium.  
\end{abstract}

\date{}

\maketitle



\section{Introduction}

Metamaterials are smart materials engineered to have properties that have 
not yet been found in nature. Their interesting applications as well as the challenges in understanding their fascinating  properties have gained  a lot of attention from the scientific community in recent years. 
One important class of metamaterials consists of negative-index metamaterials, characterized by a refractive index with a negative value over some frequency range. These metamaterials were postulated and studied 
by Veselago \cite{Veselago} in the sixties,  and their existence  was confirmed by Shelby, Smith, and Schultz \cite{ShelbySmithSchultz} in 2001. New fabrication techniques, whose principles are based on the theory of homogenization (see, e.g., \cite{BBF} and references therein),  now allow for  the construction of negative-index metamaterials at scales that are interesting for applications.

One of interesting applications of negative-index metamaterials is cloaking and there are several techniques used for this purpose. The first one involves the concept of complementary media which uses an anti-object to cancel the effect of light of the cloaked object.  
This was suggested by Lai et al. \cite{LCZT} and was mathematically established for related schemes in  both acoustic and electromagnetic waves in the time-harmonic regime \cite{Ng-Negative-Cloaking, MinhLoc2, Ng-Negative-Cloaking-Maxwell}. Another technique for cloaking an object using negative-index metematerials is via anomalous localized resonance, which was proposed in  \cite{Ng-CALR-O} with its roots in \cite{NicoroviciMcPhedranMilton94, MiltonNicorovici, Ng-CALR}. In  this  cloaking method, the cloak is independent of the object. There is also a  technique to make an object smaller using  plasmonic structures, as introduced by Alu and Engheta \cite{AluEngheta}. 


This paper is on the invisibility of  a source  via anomalous localized resonance.  This  was discoreved by Milton and Nicorovici \cite{MiltonNicorovici} (see also \cite{NicoroviciMcPhedranMilton94, NMMP1}) for constant radial-shell plasmonic structures in the two-dimensional quasistatic regime.  They showed that  a dipole source is  invisible  if the distance from it to the (negative-index)  shell plasmonic structure is less than a critical value; otherwise, it is visible. A key character of this cloaking technique is that the cloaking phenomenon is relative: the invisibility takes place when the source is normalized so that the power is  bounded. It is worth noting that  for both  cloaking using complementary media and cloaking an object via anomalous localized resonance, cloaking happens for a (fixed) source away from the plasmonic structure. Milton and Nicorovici's work was later developed by various authors \cite{BouchitteSchweizer10,AmmariCiraoloKangLeeMilton, AmmariCiraoloKangLeeMilton2, KLO, KohnLu, Ng-CALR-CRAS, Ng-CALR, Ng-CALR-frequency, Ross}; see also the references therein. 
In  the quasistatic acoustic regime, a  general setting  for this type of cloaking 
was studied  in  \cite{Ng-CALR}. We introduced there the concept of doubly complementary media for a general shell,  which roughly states that the plasmonic structure is complementary to a part of the core and a part of the exterior of the core-shell structure. The invisibility via anomalous localized resonance of a general 
source  for this structure was investigated there. In particular, we showed that the invisibility occurs when the power of the plasmonic structures is blown up. These results  were later extended for the finite frequency acoustic regime in \cite{Ng-CALR-frequency}. 
It is worth noting that the character of the resonance associated with negative-index metamaterials is quite complex; two different types of resonance, localized and complete,  can occur in very similar settings \cite{MinhLoc1}. 

Though the invisibility via anomalous localized resonance of  a source has been extensively investigated for the acoustic waves,  this problem has not yet been sufficiently developed in the electromagnetic setting. The goal of this paper is to fill this gap.   To this end, we first introduce the concept of doubly complementary media for electromagnetic waves.  We then provide criteria 
for checking the invisibility of a source. Roughly speaking, we establish that $(i)$  a source is invisible if the power of the plasmonic structure is blown up (Theorem~\ref{thm-main} and the following paragraph);  $(ii)$ a source is invisible if it is sufficiently close to the plasmonic structure and is visible if it is far from the plasmonic structure (Propositions~\ref{pro2} and \ref{pro1}); $(iii)$ if the plasmonic structure is 
complementary to an annulus of constant isotropic medium,  there is a critical length that characterizes the cloaking phenomena,  as observed by Milton and Nicorovici in the acoustic quasi-static regime \cite{MiltonNicorovici} (Theorem~\ref{thm3}).    

Two difficulties in the study of the invisibility of  a source via anomalous localized resonance are as follows. Firstly, the problem is unstable. This can be explained by the fact that the equations describing the phenomena  have sign-changing coefficients; hence the ellipticity and the compactness are lost in general. Secondly, a localized resonance might appear, i.e., the field explodes in some regions and remains bounded in others as the loss goes to 0.  Our analysis  involves three-sphere inequalities and  the localized singularity removal technique introduced in \cite{Ng-Negative-Cloaking, Ng-Superlensing} plays an important role. Negative-index metamaterials have also been investigated using the knowledge of Neumann-Poincare's operator, see e.g., \cite{AKL} and the references therein. Nevertheless, to our knowledge,  the behavior of the fields cannot be addressed using this method unless the family of the eigenfunctions are somehow explicit.

\section{Statements of the main results} 

Let $\omega > 0$, and let $\Omega_1 \Subset  \Omega_2 \Subset  \mR^3$ be smooth,  bounded,  simply connected,  open subsets of $\mR^3$ \footnote{In this paper, the notation $D \Subset \Omega$ means $\bar D \subset \Omega$ for two subsets $D$ and $\Omega$ of $\mR^3$.}.  
Let $\ep, \mup$ be defined in $\mR^3 \setminus (\Omega_2 \setminus \Omega_1)$ and $\en, \mun$ be defined in $\Omega_2 \setminus \Omega_1$ such that 
$\ep, \mup$, $- \en$, and $- \mun$  are real,  symmetric,  {\bf uniformly elliptic},  matrix-valued functions.  Set, for $\delta \ge 0$,  
\begin{equation}\label{def-eDelta}
(\eps_\delta, \mu_\delta)  = \left\{\begin{array}{cl}
\en + i \delta I,  \mun + i \delta I  & \mbox{ in } \Omega_2 \setminus \Omega_1, \\[6pt]
\ep, \mup & \mbox{ in } \mR^3 \setminus (\Omega_2 \setminus \Omega_1). 
\end{array} \right. 
\end{equation}
Here and in what follows, $I$ denotes the $(3 \times 3)$ identity matrix. We also denote $B_R(x)$ as the  open ball in $\mR^3$ centered at $x \in \mR^3$ and of radius $R>0$; when $x = 0$, we simply denote $B_{R}$.   As usual, we assume  that for some $R_0 > 0$,  $\Omega_2 \subset B_{R_0}$ and  $(\ep, \mup) = (I, I)$ in $\mR^3 \setminus B_{R_0}$, and, for the application of the unique continuation principle, see \cite{Tu, BCT},   \begin{equation}
\ep, \mup, \en, \mun \mbox{ are piecewise } C^1. 
\end{equation}

Given $\delta > 0$ and $J \in [L^2(\mR^3)]^3$ with compact support, let $(E_\delta, H_\delta) \in [H_{\loc}(\curl, \mR^3)]^2$ be  the unique radiating solution of the Maxwell equations 
\begin{equation}\label{Main-eq-delta}
\left\{\begin{array}{cl}
\nabla \times E_\delta = i \omega \mu_\delta H_\delta &  \mbox{ in } \mR^3, \\[6pt]
\nabla \times H_\delta = - i \omega \eps_\delta E_\delta + J &  \mbox{ in } \mR^3.  
\end{array} \right.
\end{equation}
Recall that a solution $(E, H) \in [H_{\loc}(\curl, \R^3\setminus B_R)]^2$, for some $R> 0$, of the Maxwell equations 
\[
\begin{cases}
\nabla \times E = i \omega H  &\text{ in } \mathbb{R}^3\setminus B_R,\\[6pt]
\nabla \times H = -i \omega E  &\text{ in } \mathbb{R}^3\setminus B_R,
\end{cases}
\]
is called radiating if it satisfies one of the (Silver-M\"{u}ller) radiation conditions
\begin{equation*}
H \times x - |x| E = O(1/|x|) \quad   \mbox{ and } \quad  E\times x + |x| H = O(1/|x|) \qquad \mbox{ as } |x| \to + \infty. 
\end{equation*}
Herein, for $\alpha \in \mR$, $O(|x|^\alpha)$ denotes a quantity whose norm is bounded by $C |x|^\alpha$ for some constant $C>0$.

For an open subset $\Omega$ of $\mR^3$ of class $C^1$,  one denotes  
\begin{equation*}
H(\curl, \Omega) = \Big\{u \in [L^2(\Omega)]^3; \nabla \times u \in [L^2(\Omega)]^3 \Big\}
\end{equation*}
and 
\begin{equation*}
H_{\loc}(\curl, \Omega) = \Big\{u \in [L^2_{\loc}(\Omega)]^3; \nabla \times u \in [L^2_{\loc}(\Omega)]^3 \Big\}. 
\end{equation*}
One also denotes 
$$
\|u \|_{H(\curl \Omega)} = \| u\|_{L^2(\Omega)}  + \|\nabla \times u \|_{L^2(\Omega)} \mbox{ for } u \in H(\curl, \Omega). 
$$

Physically, $\eps_\delta$ and $\mu_\delta$ describe the permittivity and the permeability of the considered medium,   and $\omega$ is the frequency. The set $\Omega_2 \setminus \Omega_1$ is   the shell region (plasmonic structure) in which the permittivity and the permeability are negative,  and $i \delta I$ describes the loss of this plasmonic structure.   
For $\supp J \cap (\Omega_2 \setminus \Omega_1) = \emptyset$, the power $\cP_\delta(E_\delta, H_\delta)$, or more precisely the dissipation energy,  in $\Omega_2 \setminus \Omega_1$, is defined by 
\begin{equation}\label{def-power}
\cP_\delta (E_\delta, H_\delta) = \delta  \int_{\Omega_2 \setminus \Omega_1} |(E_\delta, H_\delta)|^2. 
\end{equation}

In this paper, we investigate the behavior of $(E_\delta, H_\delta)$ away from the plasmonic shell and  the behavior of  the power $\cP_\delta$ as $\delta \to 0$ in the doubly complementary setting that will be introduced in Definition~\ref{def-DCM}, wherein localized resonance can occur. 
As seen later, the behavior of $\cP_\delta$ depends strongly on the location of $\supp J$ relative to the shell $\Omega_2 \setminus \Omega_1$ (Propositions~\ref{pro2} and \ref{pro1} and Theorem~\ref{thm3}). As a consequence of our results, one derives that 
a source is relatively invisible when the power is explored as $\delta \to 0$.

We now describe the problem in more detail. Given a matrix-valued function  $A$  defined in $\Omega$, a bi-Lipschitz homeomorphism ${\mathcal T}: \Omega \to \Omega'$,  and a vector field $j$ defined in $\Omega$, the following standard notations are used, for $y \in \Omega'$: 
\begin{equation*} 
{\mathcal T}_* A(y) = \frac{\nabla {\mathcal T}  (x)  A (x) \nabla  {\mathcal T} ^{T}(x)}{\det \nabla  {\mathcal T}(x)} \quad \mbox{ and } \quad {\mathcal T}_*  j (y) = \frac{j(x)}{\det \nabla  {\mathcal T}(x)},
\end{equation*}
with $x ={\mathcal T}^{-1}(y)$.  We first recall the definition of complementary media  \cite{Ng-Superlensing-Maxwell}:

\begin{definition}[Complementary media]   \fontfamily{m} \selectfont
 \label{def-Geo} Let $\Omega_1 \Subset  \Omega_2 \Subset  \Omega_3 \Subset  \mR^3$ be smooth,  bounded,   simply connected,  open subsets of $\mR^3$. The medium in $\Omega_2 \setminus \Omega_1$ characterized by a pair of two symmetric matrix-valued functions $(\eps_1, \mu_1)$  and the medium in $\Omega_3 \setminus \Omega_2$ characterized by a pair of two symmetric, uniformly elliptic, matrix-valued functions  $(\eps_2, \mu_2)$  are said to be  {\it  complementary} if 
there exists a diffeomorphism $\cF: \Omega_2 \setminus \bar \Omega_1 \to \Omega_3 \setminus \bar \Omega_2$ such that $\cF \in C^1(\bar \Omega_2 \setminus \Omega_1)$, 
\begin{equation}\label{cond-ASigma}
(\cF_*\eps_1, \cF_*\mu_1)   = (\eps_2, \mu_2)   \mbox{ for  } x \in  \Omega_3 \setminus \Omega_2, 
\end{equation}
\begin{equation}\label{cond-F-boundary}
\cF(x) = x \mbox{ on } \partial \Omega_2, 
\end{equation}
and the following two conditions hold: 1) There exists an diffeomorphism extension of $\cF$, which is still denoted by  $\cF$, from $\Omega_2 \setminus \{x_1\} \to \mR^3 \setminus \bar \Omega_2$ for some $x_1 \in \Omega_1$; and 2) there exists a diffeomorphism $\cG: \mR^3 \setminus \bar \Omega_3 \to \Omega_3 \setminus \{x_1\}$ such that $\cG \in C^1(\mR^3 \setminus \Omega_3)$, 
$\cG(x) = x \mbox{ on } \partial \Omega_3$,
and $
\cG \circ \cF : \Omega_1  \to \Omega_3 \mbox{ is a diffeomorphism if one sets } \cG \circ \cF(x_1) = x_1.
$
\end{definition}

Conditions  \eqref{cond-ASigma} and \eqref{cond-F-boundary} are the main assumptions in Definition~\ref{def-Geo}. 
Conditions 1) and 2)  are  mild assumptions. Introducing $\cG$ makes the analysis more accessible.  The key point behind this requirement is roughly  the following property. Assume that $(E, H) \in [H(\curl, \Omega_3 \setminus \Omega_1)]^2$ is a solution of the system 
 \begin{equation*} \left\{
\begin{array}{cl}
\nabla \times E = i \omega  \mu H & \mbox{ in } \Omega_3 \setminus \Omega_1, \\[6pt]
\nabla \times H = -  i \omega \eps E & \mbox{ in } \Omega_3 \setminus \Omega_1, 
\end{array}\right. 
\end{equation*}
with $(\eps,\mu)  = (\eps_1, \mu_1)$ in $\Omega_3 \setminus \Omega_2$ and $(\eps_2, \mu_2)$ otherwise.  Set 
$$
(\hE, \hH) =(\cF*E, \cF*H) \mbox{ in } \Omega_3 \setminus \Omega_2 
$$
(see \eqref{TEH} for the notation ${\mathcal F}*$).  Then, by a change of variables (Lemma~\ref{lem-TO}),  $(\hE, \hH)$ satisfies the same Maxwell system as $(E, H)$ in $\Omega_3 \setminus \Omega_2$ and 
$$
\hE \times \nu = E \times \nu \mbox{ on } \partial \Omega_2  \quad \mbox{ and } \quad \hH \times \nu = H \times \nu \mbox{ on } \partial \Omega_2. 
$$
Since $(\eps_2, \mu_2)$ is symmetric and  uniformly elliptic,  by the unique continuation principle,  one derives that $(\hE, \hH) = (E, H)$ in $\Omega_3 \setminus \Omega_2$. 
Fields in a medium which  does not satisfy the complementary condition would be stable, i.e., $(E_\delta, H_\delta)$ is bounded away the interface $\partial (\Omega_2 \setminus \Omega_1)$, in general; see \cite{Ng-WP} for  a discussion on this topic in the acoustic case.

\medskip 
We are ready to introduce the concept of  doubly complementary media.  

\begin{definition} \label{def-DCM} \fontfamily{m} \selectfont The medium $(\eps_0, \mu_0)$ given in \eqref{def-eDelta} with $\delta  = 0$ is said to be {\it doubly complementary} if  for some $\Omega_2 \Subset  \Omega_3$, $(\ep, \mup)$ in $\Omega_3 \setminus \Omega_2$ and $(\en, \mun)$ in $\Omega_2 \setminus \Omega_1 $ are  complementary, and 
\begin{equation}\label{DCM}
(\cG_* \cF_*\ep, \cG_* \cF_*\mup) = (\ep,  \mup)  \mbox{ in } \Omega_3 \setminus \Omega_2 
\end{equation}
for some $\cF$ and $\cG$ from Definition~\ref{def-Geo}. 
\end{definition}

\begin{remark} \rm The reason complementary media satisfying \eqref{DCM} are called doubly complementary is that 
$(\en, \mun)$ in $\Omega_2 \setminus \Omega$ is not only complementary to $(\ep, \mup)$ in $\Omega_3 \setminus \Omega_2$ but also  complementary to 
$(\ep, \mup)$ in $\cF^{-1} \circ \cG^{-1} (\Omega_3 \setminus \Omega_2)$. 
\end{remark}

\begin{remark}\rm
The definition of doubly complementary media is local, and this medium can be obtained as follows. Fix an arbitrary pair $(\en, \mun)$ in $\Omega_2 \setminus \Omega_1$ and a pair of ${\cF, \cG}$ given in Definition~\ref{def-Geo}. One then determines $(\ep, \mup)$ in $\Omega_3 \setminus \Omega_2$ and in $D : = \cF^{-1} \circ \cG^{-1}(\Omega_3 \setminus \Omega_2)$ by 
$$
(\ep, \mup) = (\cF_*\en, \cF*\mun) \mbox{ in } \Omega_3 \setminus \Omega_2,
$$
$$
(\ep, \mup)  =  (\cF^{-1}* \cG^{-1} * \cF_*\en, \cF^{-1}* \cG^{-1} *\cF*\mun) \mbox{ in } D. 
$$
The choice of $(\ep, \mup)$ outside $(\Omega_3 \setminus \Omega_1)  \cup D$ is arbitrary.  
\end{remark}

Given a subset $\Omega$ of $\mR^3$, we denote  $\mathds{1}_\Omega$ as its characteristic function.  For a doubly complementary medium $(\eps_0, \mu_0)$, set
\begin{equation}\label{def-tepsmu}
(\teps, \tmu) : = \left\{\begin{array}{cl} (\ep, \mup) & \mbox{ in } \mR^3 \setminus \Omega_3, \\[6pt]
(\cG_*\cF_*\ep, \cG_*\cF_*\mup) & \mbox{ in } \Omega_3, 
\end{array}\right.
\end{equation}
\begin{equation}\label{def-tJ}
\tJ =  \mathds{1}_{\mR^3 \setminus \Omega_2}  J -  \mathds{1}_{\Omega_3 \setminus \Omega_2}  \cF_*J +  \mathds{1}_{\Omega_3} \cG_*\cF_*J \mbox{ in } \mR^3, 
\end{equation}
and let $(\tE, \tH) \in [H_{\loc}(\curl, \mR^3)]^2$ be the unique radiating solution of 
\begin{equation}\label{sys-tEH} \left\{
\begin{array}{cl}
\nabla \times \tE = i \omega \tmu \tH, & \mbox{ in } \mR^3, \\[6pt]
\nabla \times \tH = - i \omega \teps \tE + \tJ, & \mbox{ in } \mR^3. 
\end{array}\right. 
\end{equation}
Note that if $(\eps_0, \mu_0)$ is doubly complementary, then  $\teps$ and $\tmu$ are uniformly elliptic in $\mR^3$ since $\det \cF < 0$ and $\det \cG < 0$.

\medskip

The first result of this paper is on  the behavior of $(E_\delta, H_\delta)$ outside $\Omega_3$ as $\delta \to 0$:   

\begin{theorem}\label{thm-main}
Let  $0< \delta < 1$, $J \in [L^2(\mR^d)]^3$ with $\supp J \subset B_{R_0}$, and let $(E_\delta, H_\delta) \in [H_{\loc}(\curl, \mR^3)]^2$ be the unique radiating  solution of \eqref{Main-eq-delta}.  Assume that $(\eps_0, \mu_0)$ is doubly complementary. 
Then, for $R> 0$,   
\begin{equation}\label{part1}
\| (E_\delta, H_\delta)  \|_{L^2(B_R \setminus \Omega_3)} \le C_R \| J \|_{L^2(\mR^3)}
\end{equation}
for some positive constant $C_R$ that  depending on $R$ but is independent of $J$ and $\delta$.  
Moreover, 
\begin{equation}\label{part2}
(E_\delta, H_\delta)  \mbox{ converges to }  (\tE, \tH)  \mbox{ in } [L^2_{\loc}(\mR^3 \setminus \Omega_3)]^6 \mbox{ as } \delta \to 0, 
\end{equation}
where $(\tE, \tH) \in [H_{\loc}(\curl, \mR^3)]^2$ is the unique radiating solution of \eqref{sys-tEH}. 
\end{theorem}

\begin{remark} \rm Estimate~\eqref{part1} confirms that, $(E_\delta, H_\delta)$ is bounded in $L^2_{\loc}(\mR^3 \setminus \bar \Omega_3)$. 
\end{remark}

As seen later in Proposition~\ref{pro1}  and Theorem~\ref{thm3}, the power can blow up in the setting considered in Theorem~\ref{thm-main}. Nevertheless,  the fields $(E_\delta, H_\delta)$ remain bounded outside $\Omega_3$: localized resonance therefore appears.  Theorem~\ref{thm-main} implies the equivalence between the blow up of the power and the invisibility of a source in the doubly complementary setting. This can be derived as follows.  Fix $J \in [L^2(\mR^3)]^3$ with $\supp J \subset B_{R_0}$. Suppose that the power $\cP_{\delta_n}(E_{\delta_n}, H_{\delta_n})$ blows up for some $\delta_n \to 0$,  i.e.,  
\begin{equation*}
\lim_{n \to \infty}  \cP_{\delta_n}(E_{\delta_n}, H_{\delta_n})  = + \infty, 
\end{equation*}
where $(E_\delta, H_\delta) \in [H(\curl, \mR^3)]^2$ is the unique radiating solution of \eqref{Main-eq-delta}. Set 
\begin{equation*}
\hE_\delta = c_\delta E_\delta, \quad  \hH_\delta = c_\delta H_\delta, \quad \mbox{ and } \quad \hJ_\delta = c_\delta J,  
\end{equation*}
where $c_\delta = \cP_\delta(E_\delta, H_\delta)^{-1/2}$. Then
$$
\lim_{n \to + \infty} \cP_{\delta_n} (\hE_{\delta_n}, \hH_{\delta_n}) = 1 \quad \mbox{ and } \quad \lim_{n \to + \infty} \| \hJ_{\delta_n} \|_{L^2(\mR^3)} = 0. 
$$
It follows from \eqref{part1} of Theorem~\ref{thm-main} that 
\begin{equation*}
\lim_{n \to + \infty} \| (\hE_{\delta_n}, \hH_{\delta_n}) \|_{L^2(B_R \setminus \Omega_3)} = 0:  
\end{equation*}
the source is invisible for observers outside $\Omega_3$ after the renormalization to have the boundedness of the power.   Theorem~\ref{thm-main} is, to our knowledge, the first result providing the connection between the blow up of the power and the invisibility of a general source in a general setting for electromagnetic waves via anomalous localized resonance. The starting point of the proof of Theorem~\ref{thm-main} involves the reflections $\cF$ and $\cG$ in the definition of doubly complementary media. The proof also uses three-sphere inequalities  and the localized singularity removal technique in an essential way.

\medskip 
When the support of $J$ is outside $\Omega_3$, one can prove that $(E_\delta, H_\delta)$ remains bounded in $[L^2_{\loc}(\mR^3)]^6$. More precisely, we have the following  slightly more general result:

\begin{proposition} \label{pro2} Let $0< \delta < 1$,   $J_e, J_m \in [L^2(\mR^3)]^3$ with $\supp J_e, \supp J_m \subset B_{R_0}$, and let $(E_\delta, H_\delta) \in [H_{\loc}(\curl, \mR^3)]^2$ be the unique radiating  solution of 
\begin{equation}\label{Main-eq-delta-1}
\left\{\begin{array}{cl}
\nabla \times E_\delta = i \omega \mu_\delta H + J_e &  \mbox{ in } \mR^3, \\[6pt]
\nabla \times H_\delta = - i \omega \eps_\delta E + J_m &  \mbox{ in } \mR^3.  
\end{array} \right.
\end{equation}
 Assume that $(\eps_0, \mu_0)$ is doubly complementary,  
$\supp J_e \cap (\Omega_3 \setminus O) = \emptyset$,  and $\supp J_m \cap (\Omega_3 \setminus O) = \emptyset$, where $O := \cF^{-1}\circ \cG^{-1}(\Omega_2)$.  We have,  for $R>0$,  
\begin{equation}\label{pro2-p1}
\| (E_\delta, H_\delta) \|_{L^2(B_R)} \le C_R \| (J_e, J_m)\|_{L^2(\mR^3)}.   
\end{equation}
Moreover, $(E_\delta, H_\delta)$ converges weakly to $(E_0, H_0) $ in $[L^2_{\loc}(\mR^3)]^6$  the unique radiating solution in $[H_{\loc}(\curl, \mR^3)]^2$ of \eqref{Main-eq-delta-1} with $\delta  = 0$. We also have 
\begin{equation}\label{pro2-p2}
\| (E_\delta, H_\delta) - (\tE, \tH) \|_{L^2 \big(B_R \setminus \Omega_3\big)} \le C_R \delta \| (J_e, J_m) \|_{L^2(\mR^3)},  
\end{equation}
where $(\tE, \tH) \in [H_{\loc}(\curl, \mR^3)]^2$ is the unique radiating solution of 
\begin{equation}\label{sys-tEH} \left\{
\begin{array}{cl}
\nabla \times \tE = i \omega \tmu \tH + \tJ_e, & \mbox{ in } \mR^3, \\[6pt]
\nabla \times \tH = - i \omega \teps \tE + \tJ_m, & \mbox{ in } \mR^3, 
\end{array}\right. 
\end{equation}
with $\tJ_e =  \mathds{1}_{\mR^3 \setminus \Omega_3}  J_e +  \mathds{1}_{\Omega_2} \cG_*\cF_*J_e$  and $\tJ_m =  \mathds{1}_{\mR^3 \setminus \Omega_3}  J_m +  \mathds{1}_{\Omega_2} \cG_*\cF_*J_m$ in $\mR^3$. 
Here  $C_R$ denotes a positive constant that depends on $R$ but is independent of $(J_e, J_m)$ and $\delta$. 
\end{proposition}

When $J_m = J$ and $J_e = 0$, Proposition~\ref{pro2} implies that the source $J$ is invisible away the shell $\Omega_2 \setminus \Omega_1$ since $(E_0, H_0) = (\tE, \tH)$ outside $\Omega_3$.  Proposition~\ref{pro2} will be used in the proof of Theorem~\ref{thm3} below  where a critical length is revealed  for the blow up of the power 
when the plasmonic structure is complementary to an annulus of constant, isotropic medium.

\medskip 
Concerning the blow up of the power, we prove the following result for a  class of media in which the  complementary property holds only locally. 

\begin{proposition}\label{pro1} Let $0< \delta < 1$, $0 < \alpha < 1$,  $J \in [L^2(\mR^3)]^3$ with $\supp J \subset B_{R_0}$,  and let $(E_\delta, H_\delta) \in [H_{\loc}(\curl, \mR^3)]^2$ be the unique radiating  solution of \eqref{Main-eq-delta}.   Assume that there exists  a diffeomorphism $\cF: \Omega_2 \setminus \Omega_1 \to \Omega_3 \setminus \Omega_2$ for some $\Omega_2 \Subset  \Omega_3 \Subset \mR^3$ such that $\cF \in C^1(\bar \Omega_2 \setminus \Omega_1)$,  $\cF(x) = x$ on $\partial \Omega_2$, and 
\begin{equation*}
(\ep, \mup) = (\cF_*\en, \cF_*\mun) \mbox{ in } D \mbox{ where } D : = B_{r_0}(x_0) \cap (\Omega_3 \setminus \Omega_2)  
\end{equation*}
for some $x_0 \in \partial \Omega_2$ and $r_0> 0$, and 
assume that $(\eps_0, \mu_0)$ is of class $C^2$ in $\overline D$.  There exists $0 < \tau_0 < r_0$, depending only on $(\ep, \mup)$ in $D$ and $D$ such that  
if \Big($D_1 := D \cap B_{\tau_0}(x_0)$ and the system \begin{equation}\label{pro1-g-S}
\left\{\begin{array}{cl}
\nabla \times E = i \omega \mu H \mbox{ in } D_1, \\[6pt] 
 \nabla \times H = - i \omega \eps E +  J - \cF_*J  \mbox{ in } D_1, \\[6pt] 
E \times \nu = H \times \nu = 0  \mbox{ on } \partial D_1\setminus  \partial B_{\tau_0}(x_0),
\end{array}\right.
\end{equation}
has no solution  solution in $[H(\curl, D_1)]^2$  \Big), then 
\begin{equation}\label{pro1-p1}
\limsup_{\delta \to 0} \delta^{2 \alpha}  \| (E_\delta, H_\delta)\|_{L^2(\Omega_2 \setminus \Omega_1)}^2 = + \infty. 
\end{equation}
\end{proposition}

Applying Proposition~\ref{pro1} with $\alpha = 1/ 2$, one obtains conditions for the blow-up of the power. 
Here is a
more quantitative  result on the blow up of the power under some additional requirements on $(\eps_0, \mu_0)$: 

\begin{theorem} \label{thm3} Let $0< \delta< 1$, $0< \alpha < 1$, and  $J \in [L^2(\mR^3)]^3$ with compact support,   and let $(E_\delta, H_\delta) \in [H_{\loc}(\curl, \mR^3)]^2$ be the unique radiating  solution of \eqref{Main-eq-delta}. Assume that $(\eps_0, \mu_0)$ is a doubly complementary medium with  $\Omega_2 = B_{r_2}$ and $\Omega_3 = B_{r_3}$ for  some $0 < r_2 < r_3$,  $(\eps, \mu) = (\lambda I, \lambda I)$ in $B_{r_3} \setminus B_{r_2}$ for some $\lambda > 0$, and 
\begin{equation}\label{thm3-support}
\supp J  \cap \cF^{-1} \circ \cG^{-1}  (\Omega_3 \setminus \Omega_2)  = \emptyset.
\end{equation}
We have 
\begin{enumerate}

\item[(1)] If  there exists a solution  $(E, H) \in [H(\curl, B_{r_0} \setminus B_{r_2})]^2$  of 
\begin{equation}\label{thm3-sys1}
\left\{\begin{array}{cl}
\nabla \times E = i \lambda  \omega H \mbox{ in } B_{r_0} \setminus B_{r_2}, \\[6pt]
\nabla \times H = - i \lambda  \omega H + J - \cF_*J  \mbox{ in } B_{r_0} \setminus B_{r_2}, \\[6pt]
E \times \nu = H \times \nu = 0 \mbox{ on } \partial B_{r_2}, 
\end{array}\right.
\end{equation}
for some $r_0 > r_2^\alpha r_3^{1 - \alpha}$ and $\alpha \le 1/2$,  then 
\begin{equation}\label{thm3-A1}
\limsup_{\delta \to 0} \delta^{2 \alpha} \| (E_\delta, H_\delta) \|_{L^2(B_{r_3})}^2 < + \infty. 
\end{equation}

\item[(2)] If  there does {\bf not} exist a solution  $(E, H) \in [H(\curl, B_{r_0} \setminus B_{r_2})]^2$ 
of \eqref{thm3-sys1} for some $r_0 <  r_2^\alpha r_3^{1 - \alpha}$,   then 
\begin{equation}\label{thm3-A2}
\limsup_{\delta \to 0} \delta^{2 \alpha} \| (E_\delta, H_\delta) \|_{L^2(B_{r_3} \setminus B_{r_2})}^2= + \infty. 
\end{equation}
\end{enumerate}
\end{theorem}

Applying Theorem~\ref{thm3} with $\alpha  = 1/2$, one derives, in the setting of Theorem~\ref{thm3},  that there exists a critical length $r_* = \sqrt{r_2 r_3}$ such that the existence or non-existence of a solution of the Cauchy problem \eqref{thm3-sys1} implies the blow up of the power for a sequence of $\delta_n \to 0$.

\medskip 
The paper is organized as follows. In Section~\ref{sect-pre}, we collect and establish several results  used in the proof of Theorem~\ref{thm-main} and Propositions~\ref{pro2} and  \ref{pro1}. The proof of Theorem~\ref{thm-main} and Propositions~\ref{pro2} and \ref{pro1} are given in Section~\ref{sect-4}. The proof of Theorem~\ref{thm3} is  given in Section \ref{sect-6}.

%
%
%
%
%
%
%
%
%
%
%


\section{Preliminaries} \label{sect-pre}

In this section, we collect and establish several facts on Maxwell equations that  are used in the proof of the main results. Let $\Omega$ be an open, connected, bounded   subset of $\mR^3$ of class $C^1$,  and  set $\Gamma = \partial \Omega$. Here and in what follows,  one  denotes 
\begin{equation*}
H^{-1/2}(\dive_\Gamma, \Gamma): = \Big\{ \phi \in [H^{-1/2}(\Gamma)]^3; \; \phi \cdot \nu = 0 \mbox{ and } \dive_\Gamma \phi \in H^{-1/2}(\Gamma) \Big\},
\end{equation*}
\begin{equation*}
\| \phi\|_{H^{-1/2}(\dive_\Gamma, \Gamma)} : = \| \phi\|_{H^{-1/2}(\Gamma)} +  \| \dive_\Gamma \phi\|_{H^{-1/2}(\Gamma)}.
\end{equation*}

We now provide a stability result for  \eqref{Main-eq-delta}.

\begin{lemma} \label{lem-stability} Let $0< \delta < 1$,  and let $(\eps_\delta, \mu_\delta)$ be defined in \eqref{def-eDelta}. Given 
$(J_e, J_m) \in [L^2(\mR^3)]^2$ with compact support in $K \subset B_{R_0}$, let $(E_\delta, H_\delta) \in [H_{\loc}(\curl, \mR^3)]^6$ be the unique radiating  solution of 
\begin{equation*}\left\{
\begin{array}{cl}
\nabla \times E_\delta = i \omega \mu_\delta H_\delta + J_e & \mbox{ in } \mR^3,\\[6pt]
\nabla \times H_\delta = -  i \omega \eps_\delta H_\delta + J_m & \mbox{ in } \mR^3.
\end{array}\right.
\end{equation*}
We have, for $R>0$,  
\begin{equation}\label{lem-stability-p1}
\|(E_\delta, H_\delta) \|_{L^2(B_R)}^2 \le C_R \Big( \delta^{-1} \| (J_e, J_m)\|_{L^2(\mR^3)} \| (E_\delta, H_\delta)\|_{L^2(K)} + \| (J_e, J_m)\|_{L^2(\mR^3)} \Big). 
\end{equation}
In particular, the following inequality holds:
\begin{equation}\label{lem-stability-p2}
\|(E_\delta, H_\delta) \|_{L^2(B_R)} \le C_R \delta^{-1} \| (J_e, J_m)\|_{L^2(\mR^3)}. 
\end{equation}
Here $C_R$ denotes a positive constant that  depends on $R$ but is  independent of $(J_e, J_m)$ and $\delta$.
\end{lemma}

\begin{proof} The proof of this lemma is quite simple. An integration by parts gives 
\begin{equation*}
\int_{B_R} \langle \mu_\delta^{-1} \nabla \times E_\delta, \nabla \times E_\delta \rangle - \omega^2 \int_{B_R} \langle \eps_\delta E_\delta, E_\delta \rangle = \int_{B_R} \mu_\delta^{-1} J_e \nabla \times \bar E_\delta + i \omega \int_{B_R} J_m \bar E_\delta  + i \omega \int_{\partial B_R} \nu \times \bar E_\delta H_\delta. 
\end{equation*}
Letting $R \to + \infty$, using the radiation condition, and considering the imaginary part, we have 
\begin{equation}\label{lem-stability-1}
\| (E_\delta, H_\delta) \|_{L^2(\Omega_2 \setminus \Omega_1)}^2  \le \frac{C}{\delta} \|(J_e, J_m) \|_{L^2(\mR^3)} \| (E_\delta, H_\delta) \|_{L^2(K)}.  
\end{equation}
By the trace theory, see, e.g.,  \cite{AV, BC}, 
$$
\| (E_\delta \times \nu, H_\delta \times \nu) \|_{H^{-1/2}(\dive_\Gamma, \partial \Omega_2 \cup \partial \Omega_1)} \le C 
\| (E_\delta, H_\delta) \|_{H(\curl, \Omega_2 \setminus \Omega_1)}. 
$$
It follows that 
\begin{equation}\label{lem-stability-2}
\| (E_\delta \times \nu, H_\delta \times \nu) \|_{H^{-1/2}(\dive_\Gamma, \partial \Omega_2 \cup \partial \Omega_1)}^2  \le \frac{C}{\delta} \|(J_e, J_m) \|_{L^2(\mR^3)} \| (E_\delta, H_\delta) \|_{L^2(K)}
\end{equation}
for some positive constant $C$ independent of $(J_e, J_m)$ and $\delta$.
As a consequence of  the unique continuation principle for Maxwell's equations, see, e.g., \cite[Lemma 3]{Ng-Superlensing-Maxwell}, one has   
\begin{equation}\label{lem-stability-3}
\| (E_\delta, H_\delta) \|_{L^2(\Omega_1)}  \le  C \Big( \| (E_\delta \times \nu, H_\delta \times \nu) \|_{H^{-1/2}(\dive_\Gamma, \partial \Omega_1)} + \|(J_e, J_m) \|_{L^2(\Omega_1)} \Big),
\end{equation}
and by the stability of the exterior problem, see, e.g., \cite[Lemma 5]{Ng-Superlensing-Maxwell}, we obtain 
\begin{equation}\label{lem-stability-4}
\| (E_\delta, H_\delta) \|_{L^2(B_R \setminus \Omega_2)}  \le  C_R \Big( \| E_\delta \times \nu \|_{H^{-1/2}(\dive_\Gamma, \partial \Omega_2)} + \|(J_e, J_m) \|_{L^2(B_R \setminus \Omega_2)} \Big)
\end{equation}
for some positive constants $C$ and $C_R$ independent of $(J_e, J_m)$ and $\delta$. 

Assertion~\eqref{lem-stability-p1} now follows from \eqref{lem-stability-1}, \eqref{lem-stability-2}, \eqref{lem-stability-3}, and \eqref{lem-stability-4}. 

Assertion~\eqref{lem-stability-p2} is a direct consequence of Assertion~\eqref{lem-stability-p1} and H\"older's inequality. 
\end{proof}

It is known that the Maxwell equations can be reduced to weakly coupled second order elliptic
equations. More precisely, let $\Omega$ be an open subset of $\mR^3$. If $(E, H) \in H^1(\Omega)$ and 
\begin{equation*}
\left\{\begin{array}{cl}
\nabla \times E = i  \omega \mu H & \mbox{ in } \Omega, \\[6pt]
\nabla \times H = - i \omega \eps E & \mbox{ in } \Omega, 
\end{array}\right.
\end{equation*}
then, for $1 \le a \le 3$,  
\begin{equation}\label{eq-H}
\dive (\mu \nabla \cH_a)  + \dive (\partial_a \mu \cH - i k \mu \epsilon^a \eps \cE) = 0 \mbox{ in } \Omega, 
\end{equation}
\begin{equation}\label{eq-E}
\dive(\eps \nabla \cE_a) + \dive (\partial_a \eps \cE + i k \eps \epsilon^a \mu \cH) = 0 \mbox{ in } \Omega. 
\end{equation}
Here $\cE_a$ and $\cH_a$  denote the $a$ component of $\cE$ and $\cH$,  respectively, and the $bc$ component $\epsilon^a_{bc}$ $(1 \le b, c \le 3)$  of $\epsilon^a$ $(1 \le a \le 3)$ denotes the usual Levi Civita permutation, i.e., 
\begin{equation*}
\epsilon^a_{bc} = \left\{\begin{array}{cl} \mbox{sign } (abc) & \mbox{ if  $abc$  is a permuation}, \\[6pt]
0 & \mbox{otherwise}. 
\end{array}\right.
\end{equation*}

Using the three-sphere inequality for elliptic systems, see,  e.g., \cite{ARRV}, we have 
 
\begin{lemma} \label{lem-3spheres} Let $\Omega$ be an open, connected, bounded   subset of $\mR^3$ of class $C^1$, and let $\eps, \, \mu$ be two elliptic, symmetric, matrix-valued functions defined in $\Omega$ of class $C^2$. Assume that 
\begin{equation*}
\left\{\begin{array}{cl}
\nabla \times E = i  \omega \mu H + J_{e} & \mbox{ in } \Omega, \\[6pt]
\nabla \times H = - i \omega \eps E + J_{m} & \mbox{ in } \Omega. 
\end{array}\right.
\end{equation*}
Let $K \Subset \Omega$ and $B_{r_0}(x_0) \subset \Omega$ for some $x_0 \in \Omega$ and $r_0 > 0$. We have 
\begin{equation*}
\| (\cE, \cH)\|_{L^2(K)} \le C \Big( \| (\cE, \cH)\|_{L^2(B_{r_0}(x_0))} + \| (J_{e}, J_{m})\|_{L^2(\Omega)} \Big)^\alpha \Big( \| (\cE, \cH)\|_{L^2(\Omega)} +  \| (J_{e}, J_{m})\|_{L^2(\Omega)}\Big)^{1 - \alpha} 
\end{equation*}
for some $0< \alpha < 1 $ and for some $C>0$ independent of $(E, H)$ and $(J_e, J_m)$. 
\end{lemma}


We end this section by recalling the following known change-of-variable  formula for  the Maxwell equations. 
\begin{lemma} \label{lem-TO}
Let $\Omega, \Omega'$ be two bounded, connected, open subsets of $\mR^3$ and ${\mathcal T}: \Omega \to \Omega'$ be  bijective such that 
${\mathcal T} \in C^1(\bar \Omega)$ and ${\mathcal T}^{-1} \in C^1(\bar \Omega')$. Assume that $\eps, \mu \in [L^\infty(\Omega)]^{3 \times 3}$,  $j \in [L^2(\Omega)]^3$,  and $(E, H) \in [H(\curl, \Omega)]^2$ is a solution to 
\begin{equation*}\left\{
\begin{array}{lll}
\nabla \times E  &=  i \omega \mu H & \mbox{ in } \Omega, \\[6pt]
\nabla \times H &=  - i \omega \eps  E + j & \mbox{ in } \Omega. 
\end{array}\right. 
\end{equation*}
Define $(E', H')$ in $\Omega'$ as follows:
\begin{equation}\label{TEH}
E'(x') = {\mathcal T}*E(x'): = \nabla {\mathcal T}^{-T}(x) E(x) \mbox{  and  } H'(x') = {\mathcal T}*H(x'): = \nabla {\mathcal T}^{-T}(x) H(x),  
\end{equation}
with $x' = {\mathcal T}(x)$.   Then $(E', H')$ is a solution to 
\begin{equation}\label{lem-TO-1}
\left\{
\begin{array}{lll}
\nabla' \times E'  &=  i \omega \mu' H' & \mbox{ in } \Omega', \\[6pt] 
\nabla' \times H' &=  - i \omega \eps'  E' + j' & \mbox{ in } \Omega', 
\end{array}\right. 
\end{equation}
where 
\begin{equation*}
\eps' = {\mathcal T}_*\eps, \quad \mu' = {\mathcal T}_*\mu, \quad j'= T_*j.
\end{equation*}
Additionally  assume that  $\Omega$ is of class $C^1$ and ${\bf T} = {\mathcal T} \big|_{\partial \Omega}: \partial \Omega \to \partial \Omega'$ is a diffeomorphism. We have \footnote{Here $\nu$ and $\nu'$ denote the outward unit normal  vector on $\partial \Omega$ and $\partial \Omega'$.} 
\begin{equation}\label{lem-TO-2}
\mbox{if } E \times \nu = g  \mbox{ and } H \times \nu = h \mbox{ on } \partial \Omega, \mbox{ then } E' \times \nu' = {\bf T}_*g \mbox{ and } H' \times \nu' = {\bf T}_* h  \mbox{ on } \partial \Omega', 
\end{equation}
where ${\bf T}_*$ is given in \eqref{def-T} below. 
\end{lemma}

For a tangential vector field $g$  defined in $\partial \Omega$,  we denote
\begin{equation}\label{def-T}
{\bf T}_*g(x') = \sign \cdot \frac{\nabla_{\partial \Omega} {\bf T}(x) g(x)}{|\det \nabla_{\partial \Omega}  {\bf T}(x)|}  \mbox{ with } x' = {\bf T}(x),  
\end{equation}
where $\sign := \det \nabla {\mathcal T} (x)/ |\det \nabla {\mathcal T}(x)|$ for some $x \in \Omega$.

\section{Proof of Theorem~\ref{thm-main} and Propositions~\ref{pro2} and \ref{pro1}} \label{sect-4}

In this section, we present the proof of Theorem~\ref{thm-main} and Propositions~\ref{pro2} and \ref{pro1}. The proof is based on three-sphere inequalities, the use of reflections  coming from the definition of complementary media as in  \cite{Ng-Complementary}, and the use of the localized singularity removal technique with its roots in \cite{Ng-Negative-Cloaking, Ng-Superlensing}.

\subsection{Proof of Theorem~\ref{thm-main}}  The starting point of the proof is the use of reflections $\cF$ and $\cG$ from the definition of doubly complementary media.  Set 
\begin{equation} \label{def-EH1delta}
(E_{1, \delta}, H_{1, \delta}) = (\cF*E_\delta, \cF*H_\delta) \mbox{ in } \mR^3 \setminus \Omega_3
\end{equation}
and 
\begin{equation} \label{def-EH2delta}
(E_{2, \delta}, H_{2, \delta}) =  (\cG*\cF*E_\delta, \cG*\cF*H_\delta) \mbox{ in } \Omega_3. 
\end{equation}
We also introduce $(\tE_\delta, \tH_\delta) \in [H_{\loc}(\curl, \mR^3)]^2$ as follows:
\begin{equation}\label{main-tEH}
(\tE_\delta, \tH_\delta) = \left\{\begin{array}{cl}
(E_\delta, H_\delta) & \mbox{ in } \mR^3 \setminus \Omega_3, \\[6pt]
(E_\delta, H_\delta)  - \Big[  (E_{1, \delta}, H_{1, \delta})  - (E_{2, \delta}, H_{2, \delta})  \Big] & \mbox{ in } \Omega_3 \setminus \Omega_2, \\[6pt]
(E_{2, \delta}, H_{2, \delta}) & \mbox{ in } \Omega_2.  
\end{array}\right.
\end{equation}
It follows from Lemma~\ref{lem-TO} and the definitions of $(\teps, \tmu)$ in \eqref{def-tepsmu} and of $\tJ$ in \eqref{def-tJ} that  $(\tE_\delta, \tH_\delta) \in [H_{\loc}(\curl, \mR^3)]^2$ is a radiating solution of the system  
\begin{equation*} \left\{
\begin{array}{cl}
\nabla \times \tE_\delta = i \omega \tmu \tH_\delta -  \delta  \omega \mathds{1}_{\Omega_3 \setminus \Omega_2} \cF_*I H_{1, \delta}  & \mbox{ in } \mR^3, \\[6pt]
\nabla \times \tH_\delta = -  i \omega \teps \tE_\delta  +  \delta \omega  \mathds{1}_{\Omega_3 \setminus \Omega_2} \cF_*I E_{1, \delta}  + \tJ & \mbox{ in } \mR^3. 
\end{array}\right. 
\end{equation*}
We derive from the definition of $(\tE, \tH)$ in \eqref{main-tEH}  that  $(\tE_\delta - \tE, \tH_\delta - \tH) \in [H_{\loc}(\curl, \mR^3)]^2$ is  radiating  and satisfies  
\begin{equation*} \left\{
\begin{array}{cl}
\nabla \times (\tE_\delta - \tE) = i \omega \tmu  (\tH_\delta - \tH) -  \delta \omega \mathds{1}_{\Omega_3 \setminus \Omega_2} \cF_*I H_{1, \delta} & \mbox{ in } \mR^3, \\[6pt]
\nabla \times (\tH_\delta - \tH) = -  i \omega \teps (\tE_\delta - \tE)  +  \delta \omega  \mathds{1}_{\Omega_3 \setminus \Omega_2} \cF_*I E_{1, \delta} & \mbox{ in } \mR^3.
\end{array}\right. 
\end{equation*}
Since $\teps$ and $\tmu$ are uniformly elliptic, it follows that, see, e.g., \cite[Lemma 4]{Ng-Superlensing-Maxwell}, 
\begin{equation}\label{main-p1}
\| (\tE_\delta - \tE, \tH_\delta - \tH) \|_{L^2(B_R)} \le C_R \delta \| (E_{\delta}, H_{\delta} )\|_{L^2(B_{R_0})}.
\end{equation}
Here and in what follows, without loss of generality, one assume that $\Omega_3 \subset B_{R_0}$. 
Applying Lemma~\ref{lem-stability} to $(E_\delta, H_\delta)$, we obtain 
\begin{equation}\label{main-p2}
 \| (E_\delta, H_\delta) \|_{L^2(B_{R_0})} \le C \delta^{-1} \| J\|_{L^2(\mR^3)}. 
\end{equation}
Here and in what follows in this proof, $C$ denotes a positive constant independent of $J$, $\delta$, and $R$.  Combining \eqref{main-p1} and \eqref{main-p2} yields 
\begin{equation}\label{main-p3}
\| (\tE_\delta - \tE, \tH_\delta - \tH) \|_{L^2(B_R)}  \le C_R \| J\|_{L^2(\mR^3)}. 
\end{equation}
This  in turn implies that
\begin{equation*}
\| (\tE_\delta, \tH_\delta ) \|_{L^2(B_R)}  \le C_R \| J\|_{L^2(\mR^3)}, 
\end{equation*}
which is \eqref{part1}.

We next establish \eqref{part2}. Using \eqref{part1}, without loss of generality, one might assume that $\supp J \cap \partial \Omega_2 = \emptyset$ and  $\supp J \cap \partial \Omega_1 = \emptyset$. Fix compact sets $K_1 \subset \Omega_1$,  $K_2 \subset \Omega_2 \setminus \bar \Omega_1$, and $
K_3 \subset B_{2 R_0} \setminus \bar \Omega_3$ (arbitrary).  Fix $x_0 \in \partial B_{3R_0/2}$, and set $r_0 = R_0/4$. 
Applying Lemma~\ref{lem-3spheres} with $\Omega = B_{2R_0 } \setminus \bar \Omega_2$, we have, for some $0< \alpha_3 < 1$ independent of $J$ and $\delta$, 
\begin{equation*}
\| (E_\delta, H_\delta ) \|_{L^2(K_3)} \le C\Big(\|(E_\delta , H_\delta) \|_{L^2(B_{r_0}(x_0))} + \| J \|_{L^2(\mR^3)} \Big)^{\alpha_3}  \Big( \|(E_\delta, H_\delta) \|_{L^2(B_{2R_0} \setminus \Omega_2)} + \| J \|_{L^2(\mR^3)} \Big)^{1- \alpha_3}.
\end{equation*}
We derive from \eqref{part1} and  \eqref{main-p2} that 
\begin{equation}\label{main-e1}
\| (E_\delta , H_\delta  ) \|_{L^2(K_3)} \le C \delta^{\alpha_3-1} \| J \|_{L^2(\mR^3)}. 
\end{equation}
Similarly, using
\begin{equation*}
\| (E_\delta, H_\delta) \|_{L^2(\cF^{-1} \circ \cG^{-1} (\Omega_2))} \le C \| \tE_\delta, \tH_\delta \|_{L^2(\Omega_2)} \le C \| J \|_{L^2(\mR^3)} 
\end{equation*}
instead of \eqref{part1}, applying Lemma~\ref{lem-3spheres} with $\Omega = \Omega_1$,  and noting that $\cF^{-1} \circ \cG^{-1} (\Omega_2)) \Subset \Omega_1$,  one obtains 
\begin{equation}\label{main-e2}
\| (E_\delta , H_\delta  ) \|_{L^2(K_1)} \le C \delta^{\alpha_1-1} \| J \|_{L^2(\mR^3)}
\end{equation}
for some $0 < \alpha_1 < 1$ independent of $J$ and $\delta$.  

Using \eqref{main-p3}, we have 
\begin{equation*}
 \| (\tE_\delta, \tH_\delta)\|_{L^2(\cF(K_2))}  \le C \| J \|_{L^2(\mR^3)}.  
\end{equation*}
Since $K_1$ and $K_3$ are arbitrary, it follows from \eqref{main-e1} and \eqref{main-e2} that 
\begin{equation*}
 \| (E_{1, \delta}, H_{1, \delta})\|_{L^2(\cF(K_2))}  \le C \delta^{\alpha_2 - 1} \| J \|_{L^2(\mR^3)} 
\end{equation*}
for some $0 < \alpha_2 < 1$ independent of $J$ and $\delta$.  This implies 
\begin{equation}\label{main-e3}
 \| (E_{\delta}, H_{\delta})\|_{L^2(K_2)}  \le C \delta^{\alpha_2 - 1} \| J \|_{L^2(\mR^3)}. 
\end{equation}

Combining  \eqref{main-e1}, \eqref{main-e2}, and \eqref{main-e3} yields 
\begin{equation*}
\| (E_\delta, H_\delta)\|_{L^2(K_1 \cup K_2 \cup K_3)} \le C \delta^{\alpha - 1} \| J \|_{L^2(\mR^3)}, 
\end{equation*}
with $\alpha = \min\{ \alpha_1, \alpha_2, \alpha_3\}$.  By choosing $K_1, K_2$, and $K_3$ such that $\supp J \subset K_1 \cup K_2 \cup K_3$ (this is possible since one assumes that $\supp J \cap (\partial \Omega_1 \cup \partial \Omega_2 ) = \emptyset$), we derive from Lemma~\ref{lem-stability} that 
\begin{equation*}
\| (E_\delta, H_\delta)\|_{L^2(B_{R_0})} \le C \delta^{\alpha/2 - 1} \| J \|_{L^2(\mR^3)}. 
\end{equation*}
Assertion \eqref{part2} now follows from \eqref{main-p1}. \qed

\begin{remark} \rm The introduction of $(\tE_\delta, \tH_\delta)$ plays an important role in our analysis. These fields also play an important role in the proof of cloaking and superlensing using complementary media, see \cite{Ng-Superlensing-Maxwell, Ng-Negative-Cloaking-Maxwell}. 
\end{remark}

\subsection{Proof of Propositions~\ref{pro2}}

We use the same notations as in the proof of Theorem~\ref{thm-main}.  The key ingredient of the proof is the existence and uniqueness of  radiating solutions of \eqref{Main-eq-delta} with $\delta = 0$.  The existence proof  is then  based on the formula
\begin{equation}\label{eq-EH-0}
(E_0, H_0) = \left\{ \begin{array}{cl}
(\tE, \tH) & \mbox{ in } \mR^3 \setminus \Omega_2, \\[6pt] 
(\cF^{-1}*\tE, \cF^{-1}* \tH) & \mbox{ in } \Omega_2 \setminus \Omega_1, \\[6pt]
(\cF^{-1}*\cG^{-1}*\tE, \cF^{-1}* \cF^{-1} *\tH) & \mbox{ in } \Omega_1. 
\end{array} \right. 
\end{equation}
Using Lemma~\ref{lem-TO},  one can check that $(E_0, H_0) \in \big[H_{\loc}(\curl, \mR^3) \big]^2$ is a radiating solution of  \eqref{Main-eq-delta-1} with $\delta = 0$. Defining $(E_{1, 0}, H_{1, 0})$ and $(E_{2, 0}, H_{2, 0})$ by \eqref{def-EH1delta} and \eqref{def-EH2delta} with $\delta = 0$ and using the unique continuation principle, one can check that  if $(E_0, H_0) \in \big[H_{\loc}(\curl, \mR^3) \big]^2$ is a radiating solution of  \eqref{Main-eq-delta-1} with $\delta = 0$, then $(E_0, H_0)$ is given by  \eqref{eq-EH-0}, which implies its uniqueness.  To establish the boundedness of $(E_\delta, H_\delta)$, we note that 
\begin{equation*} \left\{
\begin{array}{cl}
\nabla \times (E_\delta - E_0) = i \omega \mu_\delta  (H_\delta - H_0)  +  i \omega (\mu_0 - \mu_\delta) H_0 & \mbox{ in } \mR^3, \\[6pt]
\nabla \times (H_\delta - H_0) = -  i \omega \eps_\delta (E_\delta - E_0)   - i \omega (\eps_0 - \eps_\delta) E_0& \mbox{ in } \mR^3.
\end{array}\right. 
\end{equation*}
Applying Lemma~\ref{lem-stability}, one obtains 
\begin{equation*}
\| (E_\delta - E_0, H_\delta - H_0 ) \|_{L^2(B_R)}  \le C_R \|(E_0, H_0)\|_{L^2(\Omega_2 \setminus \Omega_1)},  
\end{equation*}
which yields 
\begin{equation}\label{pro2-1}
\| (E_\delta, H_\delta ) \|_{L^2(B_R)}  \le C_R \| (J_e, J_m) \|_{L^2(\mR^3)}.   
\end{equation}
The convergence of $(E_\delta, H_\delta)$ to $(E_0, H_0)$ in $\big[L^2_{\loc}\big(\mR^3 \setminus (\partial \Omega_1 \cup \partial \Omega_2) \big)\big]^6$ follows from \eqref{pro2-1} using   the compactness criterion for Maxwell equations, see,  e.g., \cite{Ng-Superlensing-Maxwell}, and the uniqueness of $(E_0, H_0)$.  

To prove \eqref{pro2-p2}, we note that 
\begin{equation*} \left\{
\begin{array}{cl}
\nabla \times (\tE_\delta - \tE) = i \omega \tmu  (\tH_\delta - \tH) - \delta \omega \mathds{1}_{\Omega_3 \setminus \Omega_2} \cF_*I H_{1, \delta} & \mbox{ in } \mR^3, \\[6pt]
\nabla \times (\tH_\delta - \tH) = -  i \omega \teps (\tE_\delta - \tE)  + \delta \omega \mathds{1}_{\Omega_3 \setminus \Omega_2} \cF_*I E_{1, \delta} & \mbox{ in } \mR^3.
\end{array}\right. 
\end{equation*}
Since $\teps$ and $\tmu$ are uniformly elliptic, we derive that  
\begin{equation*}
\| (\tE_\delta - \tE, \tH_\delta - \tH) \|_{L^2(B_R)} \le C_R \delta \| (E_{\delta}, H_{\delta} )\|_{L^2(B_{R_0})}, 
\end{equation*}
and \eqref{pro2-p2} follows from \eqref{pro2-1}. \qed

\begin{remark} \rm
Assertion~\ref{pro2-p1} and the convergence of $(E_\delta, H_\delta)$ to $(E_0, H_0)$ were established in \cite{Ng-Superlensing-Maxwell} when $\supp J_m \cap \Omega_3 = \emptyset$ and $J_e = 0$.  Assertion \eqref{pro2-p2} is known from \cite{Ng-Superlensing-Maxwell} with the term $\delta^{1/2}$
instead of $\delta$ when $\supp J \cap \Omega_3 = \emptyset$. The proof of \eqref{pro2-p2} is different from the one in \cite{Ng-CALR} because of the  involvement of the localized singularity removal technique, even for a stable situation (see also \cite{Ng-survey} for the acoustic setting). 
\end{remark}

As a consequence of Proposition~\ref{pro1}, one obtains the following result that  is used in the proof of Theorem~\ref{thm3}:

\begin{corollary}\label{cor-pro2}
Assume that $(\eps_0, \mu_0)$ is doubly complementary. Let $J_e, J_m \in [L^2(\mR^3)]^3$ with $\supp J_e, \supp J_m \subset B_{R_0}$,  $\supp J_e \cap (\Omega_3 \setminus O) = \emptyset$, and $\supp J_m \cap (\Omega_3 \setminus O) = \emptyset$, where $O := \cF^{-1}\circ \cG^{-1}(\Omega_2)$. Set $\Gamma = \partial (\Omega_3 \setminus O)$, 
 and  let  $f, g \in H^{-1/2}(\dive_\Gamma, \Gamma)$. Assume that  $(E_\delta, H_\delta) \in [L^2_{\loc}(\mR^3)]^6$ is  the unique radiating  solution of
\begin{equation*} 
\left\{\begin{array}{cl}
\nabla \times E_\delta = i \omega \mu_\delta H_\delta + J_e &  \mbox{ in } \mR^3 \setminus \Gamma, \\[6pt]
\nabla \times H_\delta = - i \omega \eps_\delta E_\delta + J_m &  \mbox{ in } \mR^3 \setminus \Gamma, \\[6pt]
[E_\delta \times \nu] = f, \quad [H_\delta \times \nu] = g &  \mbox{ on } \Gamma.  
\end{array} \right.
\end{equation*}
We have,  for $R>0$,  
\begin{equation}\label{pro2-p1}
\| (E_\delta, H_\delta) \|_{L^2(B_R)} \le C_R \Big( \| (J_e, J_m)\|_{L^2(\mR^3)} +  \|(f, g)\|_{H^{-1/2}(\dive_\Gamma, \Gamma)} \Big) 
\end{equation}
for some positive constant  $C_R$ independent of $(J_e, J_m)$, $(f, g)$, and $\delta$. 
\end{corollary}

\begin{proof} Since $R_0$ is fixed but arbitrary, one might assume that $\Omega_3 \subset B_{R_0/2}$.  By the trace theory, see, e.g., \cite{AV, BC}, there exists $(E, H) \in \big[H \big(\curl, (B_{R_0} \setminus \Omega_3) \cup O \big) \big]^2$ such that 
\begin{equation*}
E \times \nu = f, \quad H \times \nu = g \mbox{ on } \Gamma, 
\end{equation*}
and 
\begin{equation*}
\| (E, H) \|_{H\big(\curl, (B_{R_0} \setminus \Omega_3) \cup O \big)} \le C \| (f, g) \|_{H^{-1/2}(\dive_\Gamma, \Gamma)}
\end{equation*}
for some positive constant $C$ independent of $(f, g)$. Fix $\varphi  \in C^1(\mR^3)$ such that 
$\varphi = 1 $ in $B_{R_0/2}$ and $\supp \varphi \Subset B_{R_0}$.  
Set 
\begin{equation*}
(\cE_\delta, \cH_\delta) 
=
\left\{
\begin{array}{cl}
(E_\delta - \varphi E, H_\delta - \varphi H)  & \mbox{ in } \mR^3 \setminus (\Omega_3 \setminus O), \\[6pt]
(E_\delta, H_\delta) & \mbox{ in } \Omega_3 \setminus O. 
\end{array}\right.
\end{equation*}
Then $(\cE_\delta, \cH_\delta) \in [H_{\loc}(\curl, \mR^3)]^2$ is  radiating and satisfies  
\begin{equation*} 
\left\{\begin{array}{cl}
\nabla \times  \cE_\delta = i \omega \mu_\delta  \cH_\delta + J_{e,\delta}&  \mbox{ in } \mR^3, \\[6pt]
\nabla \times  \cH_\delta = - i \omega \eps_\delta \cE_\delta  + J_{m, \delta} &  \mbox{ in } \mR^3, 
\end{array} \right.
\end{equation*}
where
$$
J_{e, \delta} = \mathds{1}_{\mR^3 \setminus (\Omega_3 \setminus O)} \big[- \nabla \times (\varphi E) + i \omega \mu_\delta (\varphi H)\big] + J_e  \mbox{ and } J_{m, \delta} = \mathds{1}_{\mR^3 \setminus (\Omega_3 \setminus O)} \big[ -  \nabla \times (\varphi H) - i \omega \eps_\delta (\varphi E) \big] + J_m. 
$$
The conclusion now follows by applying Proposition~\ref{pro1} to $(\cE_\delta , \cH_\delta)$. 
\end{proof}

\subsection{Proof of Propositions~\ref{pro1}} 
Using a translation, local charts, and Lemma~\ref{lem-TO},  one can assume that $x_0 = 0$ and  $D = B_{r_0} \cap \mR^3_+$, where $\mR^3_+ : =  \{x \in \mR^3;  x \cdot e_3 > 0 \}$ with $e_3= (0,  0, 1)$. Set 
\begin{equation*}
(E_{1, \delta}, H_{1, \delta}) = (\cF*E_\delta, \cF*H_\delta) \mbox{ in } D
\end{equation*}
and
\begin{equation*}
(\cE_\delta, \cH_\delta) = (E_\delta, H_\delta) - (E_{1, \delta}, H_{1, \delta})  \mbox{ in } D. 
\end{equation*}
Then, by Lemma~\ref{lem-TO}, 
\begin{equation*}
\left\{\begin{array}{cl}
\nabla \times \cE_{\delta} = i \omega  \mup \cH_{\delta}  -  \delta \omega \cF_*I H_{1, \delta}  &  \mbox{ in } D,  \\[6pt]
\nabla \times \cH_{\delta} = - i \omega \ep \cE_{\delta}  + \delta \omega \cF_*I E_{1, \delta} + J -  \cF_*J  & \mbox{ in } D, \\[6pt] 
\cE_\delta \times \nu  = \cH_\delta \times \nu = 0 &  \mbox{ on } \partial \mR^3_{+} \cap B_{r_0}. 
\end{array}\right.
\end{equation*}
Our goal is to prove that $(\cE_{2^{-n}}, \cH_{2^{-n}})$ is a Cauchy sequence in $[L^2(D \cap B_{\tau_0})]^6$ for some $\tau_0 > 0$ if \eqref{pro1-p1} does not hold.  

Fix an extension of  $(\ep, \mup)$ in $B_{r_0}$ of class $C^2$ and continue to  denote this extension by $(\ep, \mup)$. 
Set, in $B_{r_0}$,  
$$
f_n = \mathds{1}_D \Big( - \omega 2^{-(n+1)} \cF_*I H_{1, 2^{-(n+1)}}  + \omega  2^{-n} \cF_*I H_{1, 2^{-n}} \Big)
$$
and 
$$
g_n = \mathds{1}_D  \Big( \omega  2^{-(n+1)} \cF_*I E_{1, 2^{-(n+1)}}  - \omega  2^{-n} \cF_*I E_{1, 2^{-n}} \Big). 
$$
Let $(\bE_n, \bH_n) \in [H(\curl, B_{r_0})]^2$ be such that 
\begin{equation*}
\left\{\begin{array}{cl}
\nabla \times \bE_{n} = i \omega  \mup \bH_{n}  + f_n  &  \mbox{ in } B_{r_0},  \\[6pt]
\nabla \times \bH_{n} = - i \omega \ep \bE_{n}  + g_n  & \mbox{ in } B_{r_0}, \\[6pt]
(\bE_n \times \nu) \times \nu +  \bH_n \times \nu = 0 \mbox{ on } \partial B_{r_0}. 
\end{array}\right.
\end{equation*}
Then 
\begin{equation}\label{pro1-0}
\| (\bE_n, \bH_n) \|_{L^2(B_{r_0})} \le C 2^{-n}\Big( \| (E_{1, 2^{-n}}, H_{1, 2^{-n}}) \|_{L^2(D)} + \| (E_{1, 2^{-n-1}}, H_{1, 2^{-n-1}}) \|_{L^2(D)} \Big). 
\end{equation}

Set 
\begin{equation}\label{def-hEH-n}
(\hE_n, \hH_n) = (\cE_{2^{-(n+1)}} - \cE_{2^{-n}} , \cH_{2^{-(n+1)}} - \cH_{2^{-n}} ) \mathds{1}_{D} -    (\bE_n,  \bH_n)  \mbox{ in } B_{r_0}. 
\end{equation}
One has 
\begin{equation*}
\left\{\begin{array}{cl}
\nabla \times \hE_{n} = i \omega  \mup \hH_{n}   &  \mbox{ in } B_{r_0},  \\[6pt]
\nabla \times \hH_{n} = - i \omega \ep \hE_{n}  & \mbox{ in } B_{r_0}. 
\end{array}\right.
\end{equation*}
Set $S = (0, 0,   -r_0/4) \in \mR^3$ and denote, for $0 < r < r_0/3$  
\begin{multline*}
\|(\hE_n, \hH_n) (\cdot - S) \|_{\bH(\partial  B_r)} \\[6pt]
 = \| (\hE_n,\hH_n) (\cdot - S) \|_{H^{1/2}(\partial  B_r)} +  \| (\ep \nabla \hE_n (\cdot - S)  \cdot \nu, \mu \nabla \hH_n (\cdot - S)  \cdot \nu)\|_{H^{-1/2}(\partial B_r)}.
\end{multline*}
Using \eqref{eq-H} and \eqref{eq-E} and applying the three-sphere inequality \cite[Lemma 7]{Ng-Negative-Cloaking-Maxwell} (see also \cite[Theorem 2]{MinhLoc2}), we have, for 
$t \in (r_0/4, r_0/3)$,  
\begin{equation}\label{pro1-1}
\| (\hE_n, \hH_n)(\cdot - S) \|_{\bH(\partial B_{t}) } \le C  \| (\hE_n, \hH_n) (\cdot - S) \|_{\bH(\partial B_{r_0/4})}^{\beta} \| (\hE_n, \hH_n)  (\cdot - S)\|_{\bH(\partial B_{r_0/3})}^{1- \beta}, 
\end{equation}
where 
$$
\beta  = \frac{R_2^{-q} - R_3^{-q}}{R_1^{-q} - R_{3}^{-q}}, 
$$
with $R_1  = r_0/4$, $R_2 =  t$, and $R_3 = r_0/3$ for some positive constant $q$ depending only on $(\ep, \mup)$ and for some positive constant $C$ independent of $n$. 
Thus, by choosing $t$ close to $r_0/4$, i.e., $R_2$ close to $R_1$, one has $\beta > (1 + \alpha)/2$. Fix such a $t$.  Set $\tau_0  = t - r_0/4$.

We now prove \eqref{pro1-p1} by contradiction for this $\tau_0$. Assume that 
\begin{equation}\label{pro1-2}
\limsup_{\delta \to 0} \delta^{2 \alpha} \| (E_\delta, H_\delta) \|_{L^2(\Omega_2 \setminus \Omega_1)}^2 < + \infty. 
\end{equation}
As in the proof of Lemma~\ref{lem-stability}, one has 
\begin{equation*}
 \| (E_\delta, H_\delta) \|_{L^2(B_{R_0})}  \le C \Big(  \| (E_\delta, H_\delta) \|_{L^2(B_{R_0})}  + \| J \|_{L^2(\mR^3)} \Big). 
\end{equation*}
It follows from \eqref{pro1-2} that 
\begin{equation}\label{pro1-3}
\limsup_{\delta \to 0} \delta^{2 \alpha} \| (E_\delta, H_\delta) \|_{L^2(B_{R_0})} < + \infty.  
\end{equation}
We derive from \eqref{pro1-0} and \eqref{pro1-3} that 
\begin{equation}\label{pro1-4}
\| (\hE_n, \hH_n) \|_{L^2(B_{r_0})}  \le C 2^{\alpha n} \quad \mbox{ and } \quad \| (\bE_n, \bH_n) \|_{L^2(B_{r_0})}  \le C 2^{(\alpha -1 ) n}. 
\end{equation}

Combining \eqref{pro1-1}  and \eqref{pro1-4} and using the definition of $(\hE_n, \hH_n)$ in \eqref{def-hEH-n} yield 
\begin{equation*}
\| (\hE_n, \hH_n)(\cdot - S) \|_{\bH(\partial B_{t}) } \le C 2^{n (\alpha - \beta)}.
 \end{equation*}
Hence $(\hE_n, \hH_n)(\cdot - S)$ is a Cauchy sequence in $[L^2(B_{t})]^6$, as is $ (\cE_{2^{-n}} , \cH_{2^{-n}} )$ in $[L^2(B_{\tau_0})]^6$. It is clear that its limit satisfies \eqref{pro1-g-S}. We have a contradiction. The proof is complete. \qed

\section{Proof of Theorem~\ref{thm3}}   \label{sect-6}

For a simple presentation, we will assume that $\lambda = \omega = 1$.  Before presenting the proof,   we introduce,  for $n \ge 0$, 
\begin{equation}\label{def-jn}
\hat j_n(t) =1 \cdot 3 \cdots (2n + 1) j_n(t) \quad \mbox{ and } \quad  \hat y_n = -  \frac{y_n(t)}{1 \cdot 3 \cdots (2n-1)} ,  
\end{equation}
where $j_n$ and $y_n$ are the spherical Bessel functions of the first and the second kind  of order $n$ respectively. Recall that (see, e.g.,  \cite[(2.37) and (2.38)]{ColtonKressInverse}),    as $n \to + \infty$, 
\begin{equation}\label{jy-n}
\hat j_n(r) = r^n \big[1 + O(1/n) \big] \quad \mbox{ and } \quad \hat y_n(r) = r^{-n-1} \big[1 + O(1/n) \big]. 
\end{equation}
In what follows, let $Y_n^m$ denote the  spherical harmonic function of degree $n$ and of order $m$,  and set 
$$
U_n^m (x) = \frac{1}{\sqrt{n(n+1)}} \nabla_{S^2} Y_n^m (x) \quad \mbox{ and } V_n^m (x) = x \times U_n^m (x) \mbox{ on } \mS^2 = \partial B_1. 
$$

We are ready to give

\subsection{Proof of Assertion~(1) of Theorem~\ref{thm3}}   Let $(\bE, \bH) \in L^2(B_{r_3} \setminus B_{r_2})$ be the unique solution to 
\begin{equation}\label{eq-bEH}
\left\{\begin{array}{cl}
\nabla \times \bE = i \omega \bH&  \mbox{ in } B_{r_3} \setminus B_{r_2}, \\[6pt]
\nabla \times \bH = - i \omega \bE + J  - \cF_*J & \mbox{ in } B_{r_3} \setminus B_{r_2}, \\[6pt]
\bE \times \nu =  0  & \mbox{ on } \partial B_{r_2}, \\[6pt] 
(\bE\times \nu) \times \nu + \bH \times \nu = 0  & \mbox{ on } \partial B_{r_3},
\end{array}\right.
\end{equation}
and set 
\begin{equation*}
(\cE, \cH)= (E - \bE, H - \bH) \mbox{ in } B_{r_0} \setminus B_{r_2}. 
\end{equation*}
Then $(\cE, \cH) \in [H(\curl, B_{r_0} \setminus B_{r_2})]^2$ satisfies 
\begin{equation*}
\left\{\begin{array}{cl}
\nabla \times \cE = i \omega \cH&  \mbox{ in } B_{r_0} \setminus B_{r_2}, \\[6pt]
\nabla \times \cH = - i \omega \cE   & \mbox{ in } B_{r_0} \setminus B_{r_2}, \\[6pt]
\cE \times \nu =  0, \quad  \cH \times \nu = -\bH \times \nu & \mbox{ on } \partial B_{r_2}. 
\end{array}\right.
\end{equation*}
Since $\omega =1$, one then can represent ${\cE, \cH}$ as follows, in $B_{r_0} \setminus B_{r_2}$, see, e.g., \cite{Kirsch}, with $r = |x|$ and $\hat x = x/ |x|$, 
\begin{align*}
\cE(x) =  &    \sum_{n=1}^\infty \sum_{|m| \le n} \sqrt{n (n +1)}  \frac{\alpha_{1, n}^m \hj_{n}(r) + \alpha_{2, n}^m \hy_n(r)}{r} Y_n^m(\hat x) \hat x  \\[6pt]
& + \sum_{n=1}^\infty \sum_{|m| \le n}
 \frac{\left(r \big[\alpha_{1, n}^m  \hj_n(r) + \alpha_{2, n}^m \hy_n (r) \big] \right)'}{r} U_n^m (\hat x) \\[6pt]
& + \sum_{n=1}^\infty \sum_{|m| \le n}  \big[\beta_{1, n}^m \hj_n(r) + \beta_{2, n}^m \hy_n(r) \big] V_n^m(\hat x)
\end{align*}
and 
\begin{align*}
\cH (x)=  &   i \sum_{n=1}^\infty \sum_{|m| \le n} \sqrt{n (n +1)}  \frac{\beta_{1, n}^m \hj_{n}(r) + \beta_{2, n}^m \hy_n(r)}{r} Y_n^m(\hat x) \hat x  \\[6pt]
& + i \sum_{n=1}^\infty \sum_{|m| \le n}
 \frac{\left(r \big[\beta_{1, n}^m  \hj_n(r) + \beta_{2, n}^m \hy_n (r) \big]  \right)'}{r} U_n^m (\hat x) \\[6pt]
& + i  \sum_{n=1}^\infty \sum_{|m| \le n}  \big[\alpha_{1, n}^m \hj_n(r) + \alpha_{2, n}^m \hy_n(r) \big] V_n^m(\hat x). 
\end{align*}
 Using \eqref{jy-n} and the fact that $\cE \times \nu = 0$ on $\partial B_{r_2}$, we derive that
\begin{equation}\label{abn}
\alpha_{2, n}^m \sim  r_2^{2n} \alpha_{1, n}^m \quad \mbox{ and } \quad \beta_{2, n}^m \sim  r_2^{2n} \beta_{1, n}^m \mbox{ for } n \ge N, |m| \le n 
\end{equation}
for some sufficiently large $N > 1$. Here and in what follows, $a \sim b$ means that $a \le C b$ and $b \le C a $ for some positive constant $C$ independent of $a$ and $b$.  Combining  \eqref{jy-n} and \eqref{abn} yields 
\begin{multline}\label{finite-EH}
\| (\cE, \cH)\|_{L^2(B_{r_0} \setminus B_{r_2})}^2 \sim \sum_{n =1}^{N-1} \sum_{|m| \le n} \sum_{j=1}^2 n^3 \big(  |\alpha_{j, n}^{m}|^2 + |\beta_{j, n}^{m}|^2 \big) r_0^{2 n} \\[6pt]
+ \sum_{n =N}^\infty \sum_{|m| \le n} n^3 \big( |\alpha_{1, n}^{m}|^2 + |\beta_{1, n}^{m}|^2 \big) r_0^{2 n} < + \infty. 
\end{multline}

Set 
\begin{equation}\label{xil}
\xi_n =  \delta^{\alpha} (r_3 / r_0)^{n}  \mbox{ for  } n \ge 1. 
\end{equation} 
The key point of the proof is the construction of  $(\cE_\delta, \cH_\delta) \in [L^2(B_{r_3} \setminus B_{r_2})]^6$ defined from $(\cE, \cH)$ as follows:
\begin{align}\label{pro3-Ed}
\cE_\delta =  &    \sum_{n=1}^\infty \sum_{|m| \le n} \frac{1}{1 + \xi_{n}}  \sqrt{n (n +1)}  \frac{\alpha_{1, n}^m \hj_{n}(r) + \alpha_{2, n}^m \hy_n(r)}{r} Y_n^m(\hat x) \hat x  \nonumber  \\[6pt]
& + \sum_{n=1}^\infty \sum_{|m| \le n} \frac{1}{1 + \xi_{n}}
 \frac{\left( r \big[\alpha_{1, n}^m  \hj_n(r) + \alpha_{2, n}^m \hy_n (r) \big]   \right)' }{r}U_n^m (\hat x)  \nonumber \\[6pt]
& + \sum_{n=1}^\infty \sum_{|m| \le n}  \frac{1}{1 + \xi_{n}}  \big[\beta_{1, n}^m \hj_n(r) + \beta_{2, n}^m \hy_n(r) \big] V_n^m(\hat x)
\end{align}
and 
\begin{align}\label{pro3-Hd}
\cH_\delta =  &   i  \sum_{n=1}^\infty \sum_{|m| \le n}  \frac{1}{1 + \xi_{n}} \sqrt{n (n +1)}  \frac{\beta_{1, n}^m \hj_{n}(r) + \beta_{2, n}^m \hy_n(r)}{r} Y_n^m(\hat x) \hat x  \nonumber \\[6pt]
& + i\sum_{n=1}^\infty \sum_{|m| \le n} \frac{1}{1 + \xi_{n}}
 \frac{ \left( r \big[\beta_{1, n}^m  \hj_n(r) + \beta_{2, n}^m \hy_n (r) \big] \right)' }{r}  U_n^m (\hat x) \nonumber \\[6pt]
& + iÒ \sum_{n=1}^\infty \sum_{|m| \le n}   \frac{1}{1 + \xi_{n}} \big[\alpha_{1, n}^m \hj_n(r) + \alpha_{2, n}^m \hy_n(r) \big] V_n^m(\hat x). 
\end{align}
Note that $(\cE_\delta, \cH_\delta)$ is defined in $B_{r_3} \setminus B_{r_2}$ instead of in $B_{r_0} \setminus B_{r_2}$ as the one of $(\cE, \cH)$. 

It is clear from the definition of $(\cE_\delta, \cH_\delta)$  in \eqref{pro3-Ed} and \eqref{pro3-Hd} that 
\begin{equation}\label{pro3-EHd}
\left\{\begin{array}{cl}
\nabla \times \cE_\delta = i \omega \cH_\delta &  \mbox{ in } B_{r_3} \setminus B_{r_2}, \\[6pt]
\nabla \times \cH_\delta = - i \omega \cE_\delta   & \mbox{ in } B_{r_3} \setminus B_{r_2}, \\[6pt]
\cE_\delta \times \nu =  0 & \mbox{ on } \partial B_{r_2}, 
\end{array}\right.
\end{equation}
and 
\begin{multline}\label{pro3-norm-EHd}
\|(\cE_\delta, \cH_\delta) \|_{L^2(B_{r_3} \setminus B_{r_2})}^2 \sim \sum_{n =1}^{N-1} \sum_{|m| \le n} \sum_{j=1}^2 n^3 \big(  |\alpha_{j, n}^{m}|^2
+ |\beta_{j, n}^{m}|^2 \big) r_0^{2 n}  \\[6pt]  +  \sum_{n =N}^\infty \sum_{|m| \le n} \frac{1}{1 + \xi_n^2}n^3 \big( |\alpha_{1, n}^{m}|^2 + |\beta_{1, n}^{m}|^2 \big) r_3^{2 n}.
\end{multline}
Here,  to derive the boundary condition $\cE_\delta \times \nu =  0$ on $\partial B_{r_2}$, we use the fact that $\cE \times \nu =  0$ on $\partial B_{r_2}$.
By the definition of $\xi_n$  in \eqref{xil}, one has
\begin{equation}\label{thm3-p1}
 \quad r_3^{2n}/(1 + \xi_n^2) \le C \delta^{-2 \alpha} r_0^{2n} \quad \mbox{ for } n \ge 1,   
\end{equation}
and, since $r_0 \ge r_2^\alpha r_3^{1- \alpha}$ and $0 < \alpha \le 1/2$, 
\begin{equation}\label{thm3-p1-1}
\xi_n^2 r_2^{2n}/(1 + \xi_n^2) \le C \delta^{2(1 - \alpha)} r_0^{2n} \quad \mbox{ for } n \ge 1.  
\end{equation}
These above two inequalities can be derived by separately considering the case $\xi_n \le 1$ and $\xi_n \ge 1$.

Combining \eqref{pro3-norm-EHd} and  \eqref{thm3-p1}  and using \eqref{finite-EH} give 
\begin{equation*}
\| (\cE_\delta, \cH_\delta) \|_{L^2(B_{r_3 \setminus B_{r_2}})}^2 \le C \delta^{-2 \alpha}. 
\end{equation*}

Let $(\tbE, \tbH) \in [H(\curl, B_{r_3} \setminus B_{r_2})]^2$ be the unique solution to 
\begin{equation}\label{eq-tbEH}
\left\{\begin{array}{cl}
\nabla \times \tbE = i \omega \tbH&  \mbox{ in } B_{r_3} \setminus B_{r_2}, \\[6pt]
\nabla \times \tbH = - i \omega \tbE + \cF_*J  & \mbox{ in } B_{r_3} \setminus B_{r_2}, \\[6pt]
\tbE\times \nu  = 0  & \mbox{ on } \partial B_{r_2}, \\[6pt] 
(\tbE\times \nu) \times \nu + \tbH \times \nu = 0  & \mbox{ on } \partial B_{r_3}. 
\end{array}\right.
\end{equation}
Set 
$$
\Gamma =  \cF^{-1}(\partial B_{r_3}) = \partial \Omega_1. 
$$
Let $(\cE_{1, \delta}, \cH_{1, \delta}), (\cE_{2, \delta}, \cH_{2, \delta}) \in [L^2_{\loc}(\curl, \mR^d)]^6$ be, respectively, the unique radiating solutions to
\begin{equation}\label{EH1}
\left\{\begin{array}{cl}
\nabla \times \cE_{1, \delta} = i \omega \mu_\delta \cH_{1, \delta} &  \mbox{ in } \mR^3, \\[6pt]
\nabla \times \cH_{1, \delta} = - i \omega \eps_\delta \cE_{1, \delta}   & \mbox{ in } \mR^3 \setminus \partial B_{r_2}, \\[6pt]
[\cH_{1, \delta} \times \nu] = -  (\cH_\delta + \bH) \times \nu& \mbox{ on } \partial B_{r_2}, 
\end{array}\right.
\end{equation}
and, with $O := \cF^{-1}\circ \cG^{-1}(\Omega_2)$, 
\begin{equation}\label{EH2}
\left\{\begin{array}{cl}
\nabla \times \cE_{2, \delta} = i \omega  \mu_\delta \cH_{2, \delta} &  \mbox{ in } \mR^3 \setminus (\partial B_{r_3} \cup \Gamma) , \\[6pt]
\nabla \times \cH_{2, \delta} = - i \omega \eps_\delta \cE_{2, \delta} + \mathds{1}_{\mR^3 \setminus B_{r_3}} J  +  \mathds{1}_{O} J   & \mbox{ in } \mR^3 \setminus  (\partial B_3 \cup \Gamma), \\[6pt]
[\cE_{2, \delta} \times \nu] = (\cE_\delta + \bE + \tbE) \times \nu, \;  [H_{2, \delta} \times \nu] = (\cH_\delta +  \bH +  \tbH ) \times \nu  & \mbox{ on } \partial B_{r_3}.\\[6pt]
[\cE_{2, \delta} \times \nu] = \big( \cF^{-1}*\tbE \big) \times \nu, \quad [H_{2, \delta} \times \nu] = \big( \cF^{-1}* \tbH \big) \times \nu  & \mbox{ on } \Gamma. 
\end{array}\right.
\end{equation}
Using \eqref{thm3-support} and applying Lemma~\ref{lem-TO}, one can derive from \eqref{eq-bEH}, \eqref{pro3-EHd},  \eqref{eq-tbEH}, \eqref{EH1}, and \eqref{EH2}, that 
\begin{multline}\label{EH-delta-**}
(E_\delta, H_\delta) - (\cE_\delta + \bE + \tbE, \cH_\delta + \bH + \tbH) \mathds{1}_{B_{r_3} \setminus B_{r_2}} \\[6pt]
- \Big( \cF^{-1}* \tbE, \cF^{-1}*\tbH \Big) \mathds{1}_{\Omega_2 \setminus \Omega_1}  
= (\cE_{1, \delta}, \cH_{1, \delta}) + (\cE_{2, \delta}, \cH_{2, \delta}) \mbox{ in } \mR^3. 
\end{multline}

From the definitions of $(\cE, \cH)$ and $(\cE_\delta, \cH_\delta)$,  we have  
\begin{multline*}
\| (\cE_\delta - \cE, \cH_\delta - \cH ) \|_{H^{-1/2}(\dive_\Gamma, \partial B_{r_2})}^2 \\[6pt] \sim   \sum_{n =1}^{N-1} \sum_{|m| \le n} \sum_{j=1}^2 \frac{\xi_n^2}{1 + \xi_{n}^2}  n^3 \big(  |\alpha_{j, n}^{m}|^2 + |\beta_{j, n}^{m}|^2 \big) r_2^{2 n}  
+   \sum_{n=1}^\infty \sum_{|m| \le n} \frac{\xi_n^2}{1 + \xi_{n}^2} n^3 \Big( |\alpha_{1, n}^m|^2 + |\beta_{1, n}^m|^2  \Big) r_2^{2n}, 
\end{multline*}
\begin{multline*}
\| (\cE_\delta , \cH_\delta  \|_{H^{-1/2}(\dive_\Gamma, \partial B_{r_3})}^2 \\[6pt]
\sim  \sum_{n=1}^N \sum_{|m| \le n}  \sum_{j=1}^2 \frac{1}{1 + \xi_{n}^2} n^3 \Big( |\alpha_{j, n}^m|^2 + |\beta_{j, n}^m|^2  \Big) r_3^{2n}  +   \sum_{n=1}^\infty \sum_{|m| \le n} \frac{1}{1 + \xi_{n}^2} n^3 \Big( |\alpha_{1, n}^m|^2 + |\beta_{1, n}^m|^2  \Big) r_3^{2n}. 
\end{multline*}   
Using the trace theory, we derive from \eqref{finite-EH} and \eqref{thm3-p1-1} that 
\begin{equation}\label{thm3-p2}
\| (\cE_\delta - \cE, \cH_\delta - \cH ) \|_{H^{-1/2}(\dive_\Gamma, \partial B_{r_2})}^2 \le  C \delta^{2 ( 1- \alpha )} 
\end{equation}
and from \eqref{finite-EH} and \eqref{thm3-p1} that 
\begin{equation}\label{thm3-p3}
\| (\cE_\delta , \cH_\delta  \|_{H^{-1/2}(\dive_\Gamma, \partial B_{r_3})}^2  \le  C \delta^{- 2 \alpha }. 
\end{equation}

Using the fact that, on $\partial B_{r_2}$,  
$$
(\cH_\delta + \bH) \times \nu  =  (\cH_\delta - \cH) \times \nu + (\cH + \bH) \times \nu  =  (\cH_\delta - \cH) \times \nu 
$$
and applying Lemma~\ref{lem-stability}, we derive from \eqref{EH1} and \eqref{thm3-p2} that 
\begin{equation}\label{thm3-p4}
\| (\cE_{1, \delta}, \cH_{1, \delta}) \|_{L^2(B_R)} \le C_R \delta^{-\alpha}. 
\end{equation}
Applying Corollary~\ref{cor-pro2}, we obtain from \eqref{thm3-p3} that  
\begin{equation}\label{thm3-p5}
\| (\cE_{2, \delta}, \cH_{2, \delta}) \|_{B_R} \le C_R \delta^{-\alpha}. 
\end{equation}

Assertion \eqref{thm3-A1} now follows from \eqref{EH-delta-**}, \eqref{thm3-p4} and \eqref{thm3-p5}. \qed 
 
\subsection{Proof of the second statement of Theorem~\ref{thm3}} 

One of the ingredient of the proof  is the following three-sphere inequality, which is conducted in the spirit of  Hadamard's famous one  \cite{Hadamard}.

\begin{lemma} \label{lem-3spheres-ball}Let $\omega > 0$, $0< R_1 < R_2 < R_3 \le R$, and let $(J_e, J_m) \in [L^2(B_R)]^6$. Assume that $(E, H) \in [H(\curl, B_R)]^2$ is a solution of 
\begin{equation*}
\left\{\begin{array}{cl}
\nabla \times E = i \omega H + J_e \mbox{ in } B_{R} , \\[6pt]
\nabla \times H = - i \omega H + J_m  \mbox{ in } B_{R}. 
\end{array}\right.
\end{equation*}
We have, with $\alpha = \ln (R_1/R_3)  \Big/ \ln (R_2 / R_3)$,  
\begin{equation*}
\| (E, H) \|_{L^2(B_{R_2})} \le C \Big(\| (E, H) \|_{L^2(B_{R_1})} + \| (J_e, J_m) \|_{L^2(B_R)}\Big)^\alpha \Big( \| (E, H) \|_{L^2(B_{R_3} )}  + \|(J_e, J_m) \|_{L^2(B_{R})}\Big)^{1 - \alpha}  
\end{equation*} 
for some positive constant $C$ that depends on $\omega$ and $R$ but is independent of $R_1$, $R_2$, $R_3$, $(J_e, J_m)$,  and $(E, H)$. 
\end{lemma}

\begin{proof} Let $(\bE, \bH) \in [H(\curl, B_R)]^2$ be the unique solution of 
\begin{equation*}
\left\{\begin{array}{cl}
\nabla \times \bE = i \omega \bH + J_e &  \mbox{ in } B_{R} , \\[6pt]
\nabla \times \bH = - i \omega \bE + J_m  & \mbox{ in } B_{R}, \\[6pt]
(\bE \times \nu) \times \nu +  \bH \times \nu = 0 & \mbox{ on } \partial B_R. 
\end{array}\right.
\end{equation*}
We have  
\begin{equation} \label{lem-ball-1}
\| (\bE. \bH) \|_{L^2(B_R)} \le C \| (J_E, J_H)\|_{L^2(B_R)}.
\end{equation}
In this proof, $C$ denotes a positive constant depending only on $\omega$ and $R$. By considering $(E - \bE, H- \bH)$ and using \eqref{lem-ball-1}, without loss of generality, one might assume that $(J_e, J_m) = (0, 0)$ in $B_R$. 
By a standard change of variables and the use of  Lemma~\ref{lem-TO}, one might also  assume that $\omega =1$ as well. Since $\omega =1$ and $(J_e, J_m) = (0, 0)$ in $B_R$, one has, see,  e.g. \cite[Theorem 2.48]{Kirsch},  
\begin{equation*}
E(x) =      \sum_{n=1}^\infty \sum_{|m| \le n} \left( \sqrt{n (n +1)}  \frac{\alpha_{n}^m \hj_{n}(r)}{r} Y_n^m (\hat x) \hat x   + \frac{ \alpha_{n}^m   \big(r\hj_n(r) \big)'}{r} U_n^m (\hat x)  +  \beta_{n}^m \hj_n(r) V_n^m(\hat x) \right)
\end{equation*}
and 
\begin{equation*}
H(x) =    i  \sum_{n=1}^\infty \sum_{|m| \le n} \left( \sqrt{n (n +1)}  \frac{\beta_{n}^m \hj_{n}(r)}{r} Y_n^m(\hat x) \hat x   +  \frac{\beta_n^m \big(r \hj_n(r) \big)'}{r} U_n^m (\hat x)  +  \alpha_{n}^m \hj_n(r) V_n^m(\hat x) \right). 
\end{equation*}
Using \eqref{jy-n}, as in the proof of Assertion~(1) of Theorem~\ref{thm3},  one obtains 
\begin{equation*}
 \|(E, H) \|_{L^2(B_{r})}^2  \sim   \sum_{n=1}^\infty \sum_{|m| \le n} n^3 \big( |\alpha_n^m|^2 + |\beta_n^m|^2 \big) r^{2n}. 
\end{equation*}
The conclusion now follows from H\"older's inequality. 
\end{proof}

We are ready to give 

\begin{proof}[Proof of Assertion~(2) of Theorem~\ref{thm3}] Set $\delta_n = 2^{-n}$. 
We establish 
\begin{equation*}
\sup_n \delta_n^{\alpha} \| (E_{\delta_n}, H_{\delta_n})\|_{L^2(\Omega_2 \setminus \Omega_1)}  =  + \infty 
\end{equation*} 
by contradiction.  Assume that 
\begin{equation}\label{claim-IE}
\sup_n \delta_n^{\alpha} \| (E_{\delta_n}, H_{\delta_n})\|_{L^2(\Omega_2 \setminus \Omega_1)}  <  + \infty.
\end{equation} 
As in the proof of Lemma~\ref{lem-stability}, we derive from \eqref{claim-IE} that 
\begin{equation*}
\sup_{n} \delta_n^{\alpha} \| (E_{\delta_n}, H_{\delta_n})\|_{L^2(B_R)}  <  + \infty. 
\end{equation*} 
This implies,  with $(E_{1, \delta}, H_{1, \delta})  = (\cF*E_\delta, \cF*H_\delta)$, that 
\begin{equation}\label{claim-IE-contradiction}
 \sup_{n} \delta_n^{\alpha} \big( \| (E_{n}, H_n)\|_{L^2(B_{r_3} \setminus B_{r_2})}   + \| (E_{1, n}, H_{1, n}) \|_{L^2(B_{r_3} \setminus B_{r_2})}  \big) <  + \infty,  
\end{equation} 
where we denote $(E_{\delta_n}, H_{\delta_n})$ and $(E_{1, \delta_n}, H_{1, \delta_n})$
 by $(E_n, H_n) $  and $(E_{1, n}, H_{1, n})$, respectively, for notational ease.

Define 
\begin{equation*}
(\cE_n, \cH_n) = \left\{ \begin{array}{cl}
 (E_n, H_n) - (E_{1, n}, H_{1, n})  &   \mbox{ in } B_{r_3} \setminus B_{r_2}, \\[6pt]
 (0, 0) & \mbox{ in } B_{r_2}. 
 \end{array}\right. 
\end{equation*}
Then
\begin{equation*}
\left\{\begin{array}{cl}
\nabla \times \cE_{n} = i \omega  \cH_{n} -  \delta_n  \omega \cF_*I H_{1, n} \mathds{1}_{B_{r_3} \setminus B_{r_2}}  &  \mbox{ in } B_{r_3},  \\[6pt]
\nabla \times \cH_{n} = - i \omega  \cE_{n} +  \delta_n \omega \cF_*I E_{1, n}  \mathds{1}_{B_{r_3} \setminus B_{r_2}}  & \mbox{ in } B_{r_3}. 
\end{array}\right.
\end{equation*}
Set, in $B_{r_3}$,  
$$
f_n  = \Big( -   2^{-(n+1)} \omega \cF_*I H_{1, 2^{-(n+1)}}   \mathds{1}_{B_{r_3} \setminus B_{r_2}} +  2^{-n} \omega  \cF_*I H_{1, 2^{-n}} \Big) \mathds{1}_{B_{r_3} \setminus B_{r_2}}
$$ 
and 
$$
g_n =   \Big( 2^{-(n+1)}  \omega \cF_*I E_{1, 2^{-(n+1)}}   \mathds{1}_{B_{r_3} \setminus B_{r_2}} -  2^{-n} \omega  \cF_*I E_{1, 2^{-n}}  \mathds{1}_{B_{r_3} \setminus B_{r_2}} \Big) \mathds{1}_{B_{r_3} \setminus B_{r_2}}.
$$ 
Let $(\bE_n, \bH_n) \in L^2(B_{R_0})$ be such that 
\begin{equation*}
\left\{\begin{array}{cl}
\nabla \times \bE_{n} = i \omega  \bH_{n}  + f_n &  \mbox{ in } B_{r_3},  \\[6pt]
\nabla \times \bH_{n} = - i \omega  \bE_{n}  + g_n & \mbox{ in } B_{r_3}, \\[6pt]
\bE_n \times \nu \times \nu +  \bH_n \times \nu = 0 \mbox{ on } \partial B_{r_3}. 
\end{array}\right.
\end{equation*}
From \eqref{claim-IE-contradiction},  we have 
\begin{equation}\label{claim-IE-contradiction-2}
\| (\bE_n, \bH_n)\|_{L^2(B_{r_3})} \le C 2^{-n(1- \alpha)}.  
\end{equation}
Here and in what follows in this proof, $C$ denotes a positive constant independent of $n$. 
Set 
$$
(\hE_n, \hH_n) =  (\cE_{n+1} - \cE_n, \cH_{n+1} - \cH_n)  -    (\bE_n,  \bH_n)  \mbox{ in } B_{r_3}. 
$$
It follows that 
\begin{equation*}
\left\{\begin{array}{cl}
\nabla \times \hE_{n} = i \omega  \hH_{n}  &  \mbox{ in } B_{r_3},  \\[6pt]
\nabla \times \hH_{n} = - i \omega  \hE_{n}  & \mbox{ in } B_{r_3}. 
\end{array}\right.
\end{equation*}
Applying Lemma~\ref{lem-3spheres-ball}, we have 
\begin{equation*}
\|(\hE_n, \hH_n) \|_{L^2(B_{\hat r_0})} \le  C \|(\hE_n, \hH_n) \|_{L^2(B_{r_2})} ^{\beta}  \|(\hE_n, \hH_n) \|_{L^2(B_{r_3})}^{1 - \beta}. 
\end{equation*}
with $r_0 < \hat r_0 < r_2^\alpha r_3 ^{1 - \alpha}$ and  $\beta = \ln (\hat r_0/ r_3)  \big/ \ln (r_2/ r_3) >  \alpha$. 
Using  \eqref{claim-IE-contradiction} and \eqref{claim-IE-contradiction-2}, we obtain 
\begin{equation*}
\|(\hE_n, \hH_n) \|_{L^2(B_{\hat r_0})} \le C 2^{-n [\beta (1 - \alpha) - \alpha  (1  -\beta)  ]} \le C 2^{n (\alpha - \beta)}.  
\end{equation*}
One derives that 
\begin{equation*}
(\hE_n, \hH_n) \mbox{ is a Cauchy sequence in } L^2(B_{\hat r_0}). 
\end{equation*}
Considering the system of the limit of $(\hE_n, \hH_n)$, one reaches a contradiction with a non-existence of a solution of \eqref{thm3-sys1}.  
\end{proof}

\end{document}